\documentclass[12pt]{article}
\usepackage{amsmath,amsthm,amssymb,hyperref}
\newcommand{\remove}[1]{}
\newtheorem*{theoremNoNum}{Theorem}
\newtheorem{theorem}{Theorem}[section]
\newtheorem{thm}{Theorem}[section]

\newtheorem{lemma}[thm]{Lemma}

\newtheorem{definition}[thm]{Definition}

\newtheorem{corollary}[thm]{Corollary}

\newtheorem{introthm}{Theorem}

\newtheorem{remark}[thm]{Remark}

\def\E{{\mathbb{E}}}
\def\Z{{\mathbb{Z}}}
\def\F{{\mathbb{F}}}
\def\R{{\mathbb{R}}}
\def\bias{{\mathrm{bias}}}
\def\poly{{\mathrm{poly}}}
\def\rank{{\mathrm{rank}}}
\def\chr{{\mathrm{char}}}
\newcommand{\eqdef}{\triangleq}
\newcommand{\ignore}[1]{}

\begin{document}

\title{On the Structure of Cubic and Quartic Polynomials}

\author{Elad Haramaty\thanks{Faculty of Computer Science, The Technion, Haifa,
Israel.
 {\tt eladh,shpilka@cs.technion.ac.il}. Research
supported by the Israel Science Foundation (grant number 439/06).}
\and Amir Shpilka\footnotemark[1]}

\maketitle

\begin{abstract}
In this paper we study the structure of polynomials of degree
three and four that have high bias or high Gowers norm, over
arbitrary prime fields. In particular we obtain the following
results.
\begin{enumerate}
\item\label{abs:bias 3} Let $f$ be a degree three polynomial with
$\bias(f)=\delta$ then there exist $r=O(\log(1/\delta))$ quadratic
polynomials $\{q_i\}$, $c=O(\log^4(\frac{1}{\delta}))$ linear
functions $\{\ell_i\}$ and a degree three polynomial $g$ such that
$f = \sum_{i=1}^r \ell_i\cdot q_i + g(\ell_1,\ldots,\ell_c)\;.$
This result generalizes the corresponding result for quadratic
polynomials.

\item\label{abs:bias 4} Let $\deg(f)=4$ and $\bias(f)=\delta$.
Then $f = \sum_{i=1}^{r} \ell_i \cdot g_i + \sum_{i=1}^r q_i \cdot
q'_i\;,$ where $r=\poly(1/\delta)$, the $\ell_i$-s are linear, the
$q_i$-s are quadratics and the $g_i$-s are cubic.

\item\label{abs:gowers 4} Let $\deg(f)=4$ and
$\|f\|_{U^4}=\delta$. Then there exists a partition of a subspace
$V \subseteq \F^n$, $\dim(V)\geq n-O(\log(1/\delta))$, to
subspaces $\{V_\alpha\}$, such that $\forall
\alpha\;\dim(V_\alpha) \geq n/\exp(\log^2(1/\delta))$ and
$\deg(f|_{V_\alpha})=3$.
\end{enumerate}

Items~\ref{abs:bias 3},\ref{abs:bias 4} extend and improve
previous results for degree three and four polynomials
\cite{KaufmanLovett08,GreenTao07}. Item~\ref{abs:gowers 4} gives a
new result for the case of degree four polynomials with high $U^4$
norm. It is the first case where the inverse conjecture for the
Gowers norm fails \cite{LMS08,GreenTao07}, namely that such an $f$
is not necessarily correlated with a cubic polynomial. Our result
shows that instead $f$ equals a cubic polynomial on a large
subspace (in fact we show that a much stronger claim holds).

\ignore{ Results~\ref{abs:bias 3},\ref{abs:bias 4} improves and
extends previous results of \cite{KaufmanLovett08}. In particular,
previous to our results it was known that $f=F(g_1,\ldots,g_c)$
where $\deg(g_i)<\deg(f)$, but $c$ was much larger and no nice
structural theorem was known. In particular for degree three
polynomials $c_3 = \exp(1/\delta)$ and for degree four  $c_4$ is a
tower of height $\exp(1/\delta)$. Result~\ref{abs:gowers 3} is a
special case of Gowers inverse theorem for $U^3$. In general it is
known \cite{Samorodnitsky07,GreenTao07} that if $\|f\|_{U^3}$ is
high then $f$ is correlated with a quadratic. Here we get a more
accurate result for degree three polynomials.
Result~\ref{abs:gowers 4} treats the case of degree four
polynomials with high $U^4$ norm. It is the first case where the
inverse conjecture for the Gowers norm fails
\cite{LMS08,GreenTao07}, namely that such an $f$ is not
necessarily correlated with a cubic polynomial. Our result shows
that instead $f$ equals a cubic polynomial on a large subspace (in
fact a much stronger result holds). }

Our techniques are based on finding a structure in the space of
partial derivatives of $f$. For example, when $\deg(f)=4$ and $f$
has high $U^4$ norm we show that there exist quadratic polynomials
$\{q_i\}_{i \in [r]}$ and linear functions $\{\ell_i\}_{i \in
[R]}$ such that (on a large enough subspace) every partial
derivative of $f$ can be written as $\Delta_y(f) = \sum_{i=1}^{R}
\ell_i \cdot q_i^y + \sum_{i=1}^r q_i \cdot \ell_i^y + q_0^y$,
where $\ell_i^y,q_i^y$ depend on $y$, the direction of the partial
derivative, $r=O(\log^2(1/\|f\|_{U_4}))$ and $R=\exp(r)$.

\end{abstract}

\thispagestyle{empty}

\newpage

\pagenumbering{arabic}

\section{Introduction}

Assume that we are given a degree $d$ polynomial $f$ that, in some
sense, `behaves' differently from a random degree $d$ polynomial.
Is there anything that we can deduce about the structure of $f$
just by knowing this fact? Recently this question received a lot
of attention, where the `behavior' of $f$ was examined with
respect to its bias or the more general notion of the Gowers norm.

\begin{definition}\label{def:bias}
Let $f:\F^n \to \F$ be a function. The bias of $f$ is defined as
$$\bias(f) = \left|\E_{\bar{a} \in \F^n}[\omega^{f(\bar{a})}]\right| \;,$$
where $\omega = e^{\frac{2 \pi i}{|\F|}}$ is a complex primitive
root of unity of order $|\F|$.
\end{definition}

Intuitively, the bias of $f$ measures how far is the distribution
induced by $f$ from the uniform distribution. We expect a random
polynomial to have a vanishing small bias (as a function of the
number of variables), so it is interesting to know what can be
said when the bias is not too small. Indeed, Green and Tao
\cite{GreenTao07} showed that if $f$ is a degree $d$ polynomial
over $\F$, such that $d < |\F|$, and $\bias(f)=\delta$ then $f$
can be written as a function of a small number of lower degree
polynomial. Formally, $f(x) = F(g_1,\ldots,g_{c_d})$ for some
function $F$ and $c_d=c_d(\bias(f),|\F|)$ polynomials $\{g_i\}$
satisfying $\deg(g_i)<d$. Note that $c_d=c_d(\bias(f),|\F|)$ does
not depend on the number of variables, i.e. it is some constant.
This result was later extended by Kaufman and Lovett
\cite{KaufmanLovett08} to arbitrary finite fields (i.e. without
the restriction $d<|\F|$). Thus, if $f$ has a noticeable bias,
unlike a random degree $d$ polynomial, then $f$ is in fact very
far from being random; simple counting arguments show that most
degree $d$ polynomials cannot be represented as functions of a few
lower degree polynomials. This result is also interesting as it
gives an average case - worst case reduction. Namely, if $f$ has
correlation $\delta$ with a lower degree polynomial then it is a
function of a small number of lower degree polynomials. One
drawback of the results of \cite{GreenTao07,KaufmanLovett08} is
the dependance of the number of lower degree polynomials on the
bias of $f$. In particular when $\deg(f)=3$,
\cite{GreenTao07,KaufmanLovett08} get the bound $c_3 =
\exp(\poly(1/\bias(f)))$ and for $\deg(f)=4$ they
bound\footnote{These numbers are not explicitly computed there,
but this is what the recursive arguments in the papers imply.}
$c_4$ by a tower of height $c_3$. On the other hand if $\deg(f)=2$
and $\bias(f) = \delta$ then it is known that $f$ can be written
as a function of at most $2\log(1/\delta)+1$ linear functions.
This can be immediately deduced from the following well known
theorem.


\begin{theorem}[Structure of quadratic polynomials]\label{thm:Dickson}
(Theorems 6.21 and 6.30 in \cite{LN97}). For every quadratic
polynomial $f: \F^n \to \F$ over a prime field $\F$ there exists
an invertible linear transformation $T$, a linear polynomial
$\ell$, and field elements $\alpha_1, \ldots \alpha_n$ (some of
which may be $0$) such that:
\begin{enumerate}
\item If $\chr(\F) = 2$ then $(q \circ T)(x) = \sum_{1=i}^{\lfloor
n/2\rfloor} \alpha_i \cdot x_{2i-1} \cdot x_{2i} + \ell(x)$,

\item If $\chr(\F)$ is odd then $(q \circ T)(x) = \sum_{1=i}^n
\alpha_i \cdot x_i^2 + \ell(x)$.
\end{enumerate}
Moreover, the number of non zero $\alpha_i$-s is invariant and
depends only on $f$.
\end{theorem}

We thus see that there is a sharp contrast between the result for
quadratic polynomials and the results for polynomials of degrees
as low as three or four. We also note that the results of Kaufman
and Lovett only guarantee that $f$ can be represented as $f(x) =
F(g_1,\ldots,g_c)$ but no nice structure like the one in
Theorem~\ref{thm:Dickson} is known. It is thus an intriguing
question whether a nice structural theorem exists for biased
polynomials and what is the correct dependance of the number of
lower degree polynomials on $\deg(f)$ and $\bias(f)$.

As mentioned above, a more general measure of randomness that was
considered is the so called Gowers norm of $f$. Intuitively, the
$U^d$ Gower norm tests whether $f$ behaves like a degree $d-1$
polynomial on $d$ dimensional subspaces. To define the Gowers norm
we first define the notion of a discrete partial derivative.

\begin{definition}(Discrete partial derivative)\label{def:partial
derivative} For a function $f:\F^n \to \F$ and a direction $y \in
\F^n$ we define $\Delta_y(f)(x) \eqdef f(x+y) - f(x)$ to be the
discrete partial derivative of $f$ in direction $y$ at the point
$x$.
\end{definition}

It is not difficult to see that if $\deg(f)=d$ then for every $y$,
$\deg(\Delta_y(f)) \leq d-1$. We now define the $d$-th Gower norm
of a function $f$.

\begin{definition}[Gowers norm \cite{AKKLR,Gowers98,Gowers01}]
The $d$-th Gower norm, $U^d$, of $f$ is defined as $$\|f\|_{U^d}
\eqdef
|\E_{x,y_1,\ldots,y_d}[\omega^{\Delta_{y_d}\ldots\Delta_{y_1}(f)(x)}]|^{1/2^d}\;,$$
where again $\omega = e^{\frac{2\pi i}{|\F|}}$.
\end{definition}

Note that $\|f\|_{U^0}=\|f\|_{U^1}=\bias(f)$. It is also clear
that if $\deg(f)=d-1$ then $\|f\|_{U^d}=1$. For more properties of
the Gowers norm we refer the reader to
\cite{Gowers98,Gowers01,GreenTao08,Samorodnitsky07,ViolaW07}.

In \cite{AKKLR} Alon et. al. showed that if
$\|f\|_{U^d}>1-\delta$, for some small $\delta$, then $f$ can be
well approximated by a degree $d-1$ polynomial. This raises the
question whether any function that has a noticeable $U^d$ norm is
somewhat correlated with a lower degree polynomial and indeed in
\cite{Samorodnitsky07,GreenTao08} this was conjectured to be the
case. This conjecture has become known as the inverse conjecture
for the Gowers norm. Samorodnitsky \cite{Samorodnitsky07} proved
that if $\|f\|_{U^3} = \delta$ where $f:\F_2^n\to \F_2$ is an
arbitrary function, then $f$ has an exponentially high (in
$\delta$) correlation with a quadratic polynomial. Namely, there
exists a quadratic polynomial $q$ such that $\Pr_{x \in
\F_2^n}[f(x) = q(X)] \geq 1/2 + \exp(-\poly(1/\delta))$. Green and
Tao \cite{GreenTao08} obtained a similar result for fields of odd
characteristic.  These results gave an affirmative answer for the
case of the $U^3$ norm. More generally, Green and Tao proved that
if $d < |\F|$ and $f$ is a degree $d$ polynomial with a high $U^d$
norm then $f$ is indeed correlated with a lower degree polynomial
\cite{GreenTao07}. Recently, the case of large characteristic was
solved by Tao and Ziegler \cite{TaoZiegler09}.\footnote{In
fact,\cite{TaoZiegler09} only get a qualitative result. No
explicit connection is known between the Gowers norm and the
correlation with polynomials.} Using ideas from ergodic theory and
the earlier \cite{BTZ09} they proved that if $f:\F^n\to
\mathcal{D}$ (where $\mathcal D$ is the unit disk in $\mathbb{C}$)
is a function with high $U^d$ norm and $d\leq |\F|$ then $f$ is
correlated with a degree $d-1$ phase polynomial.\footnote{A degree
$d-1$ phase polynomial is a function of the form $e^{2\pi i
\theta}\omega^g$, for some degree $d-1-(p-1)t$ polynomial $g$
where $\theta \in [0,1]$ and $\omega = e^{2\pi i /|\F|^t}$.} This
completely settled the conjecture for the case $d \leq |\F|$. On
the other hand, for the $U^4$ norm it was shown, independently, by
Lovett, Meshulam and Samorodnitsky \cite{LMS08} and by Green and
Tao \cite{GreenTao07} that no such result is possible when
$\F=\F_2$. Namely, \cite{LMS08} proved that the symmetric
polynomial $S_4(x_1,\ldots,x_n) \eqdef \sum_{T \subset [n], |T|=4}
\prod_{i\in T}x_i$, which is of degree four, has a high $U^4$ norm
but has an exponentially (in $n$) small correlation with any lower
degree polynomial. Similar examples where given for other fields
(when $d$ is large enough compared to the size of the field).
These examples show that for small fields the inverse conjecture
for the Gowers norm is not true in its current form. In their
work, Tao and Ziegler \cite{TaoZiegler09} proved a variant of the
conjecture for the case $d\leq |\F|$. Namely, that if a function
$f$ has high $U^d$ norm then $f$ is correlated with a phase
polynomial of a certain constant degree (but not necessarily
smaller than $d$). We note however, that if $\deg(f)=d$ then the
results of \cite{TaoZiegler09} do not give any information on $f$.
In fact, even if $\deg(f)=4$ and $f$ has a high $U^4$ norm then
nothing is known on the structure of $f$. It is thus a very
interesting question to understand the structure of low degree
polynomials having high Gowers norm over small fields.

Besides being natural questions on the own, results on the Gowers
norms had many applications in mathematics and computer science.
In his seminal work on finding arithmetic progressions in dense
sets, Gowers first defined the $U^d$ norm (for functions from
$\Z_n$ to $\Z_n$) and proved an inverse theorem for them that was
instrumental in his proofs \cite{Gowers98,Gowers01}. Bogdanov and
Viola \cite{BogdanovViola07} attempt for constructing a pseudo
random generator for constant degree polynomials relied on the
(erroneous) inverse conjecture for the Gowers norm, yet it paved
the way for other papers solving the problem
\cite{Lovett08,Viola08}. In \cite{SamorodnitskyTrevisan06}
applications of an inverse theorem for the Gowers norm to PCP
constructions was given. Samorodnitsky's proof of the inverse
theorem for the $U^3$ norm \cite{Samorodnitsky07} implies a low
degree test distinguishing quadratic functions from those that do
not have a non trivial correlation with a quadratic function. This
result also gives a test for checking the distance of a given word
from the 2nd order Reed-Muller code, beyond the list decoding
radius. For a more elaborate discussion of the connection between
additive combinatorics and computer science see \cite{Trevisan09}.

\subsection{Our results}

In this work we are able to show analogs of
Theorem~\ref{thm:Dickson} for polynomials of degree three and
four. We also prove a structural result for the case that such a
polynomial has a high Gowers norm. Our first main result is the
following.

\begin{introthm}\label{thm:intro:deg-3:bias}(biased cubic
polynomials)\sloppy Let $\F$ be a finite field and $f \in
\F[x_1,\ldots,x_n]$ a cubic polynomial ($\deg(f)=3$) such that
$\bias(f)=\delta$. Then there exist $c_1 = O(\log (1/\delta))$
quadratic polynomials $q_1,\ldots,q_{c_1} \in \F[x_1,\ldots,x_n]$
and linear functions $\ell_1,\ldots,\ell_{c_1} \in
\F[x_1,\ldots,x_n]$ and another $c_2 =
O(\log^4(\frac{1}{\delta}))$ linear functions
$\ell'_1,\ldots,\ell'_{c_2} \in \F[x_1,\ldots,x_n]$ such that $f =
\sum_{j=1}^{c_1} \ell_j \cdot q_j +
g(\ell'_1,\ldots,\ell'_{c_2})$, where $g$ is cubic.
\end{introthm}

We note that if it weren't for the $g(\ell'_1,\ldots,\ell'_{c_2})$
part then this result would be quantitatively the same as
Theorem~\ref{thm:Dickson} (and tight of course). It is an
interesting open question to decide whether we can do only with
the $\sum_{j=1}^{O(\log_{|\F|} 1/\delta)} \ell_i \cdot q_i$ part.
Using the same techniques we show a similar result for the case
that $\|f\|_{U^3}
> \delta$.

\begin{introthm}\label{thm:intro:deg-3:gowers}(cubic polynomials with high $U^3$
norm)\sloppy Let $\F$ be a finite field and $f \in
\F[x_1,\ldots,x_n]$ a cubic polynomial such that
$\|f\|_{U^3}=\delta$. Then there exist $c+1 =
O(\log^2(\frac{1}{\delta}))$ quadratic polynomials
$q_0,\ldots,q_{c} \in \F[x_1,\ldots,x_n]$ and $c$ linear functions
$\ell_1,\ldots,\ell_{c} \in \F[x_1,\ldots,x_n]$  such that $f =
\sum_{j=1}^{c} \ell_j \cdot q_j +q_0$.
\end{introthm}

Note that the difference between the structure of $f$ in
Theorems~\ref{thm:intro:deg-3:bias} and
\ref{thm:intro:deg-3:gowers} is the number of quadratic function
required. Recall that in \cite{Samorodnitsky07} Samorodnitsky
proved that if an $\F_2$ function $f$ has a high $U^3$ norm then
it has an exponentially (in $\|f\|_{U^3}$) high correlation with a
quadratic polynomial. Thus, our theorem shows that when $f$ is a
cubic polynomial then a much stronger statement holds. Namely, $f$
has correlation $\exp(\log^2(1/\delta))$ with a quadratic
polynomial, and further, has a nice structure.

Our second main result is an analog of
Theorem~\ref{thm:intro:deg-3:bias} for the case of quartic
polynomials (i.e. $\deg(f)=4$).

\begin{introthm}\label{thm:intro:deg-4:bias}(biased quartic polynomials)
Let $\F$ be a finite field and $f \in \F[x_1,\ldots,x_n]$ a
quartic polynomial ($\deg(f)=4$) such that $\bias(f)=\delta$. Then
there exist $4c = \poly(|\F|/\delta)$ polynomials
$\{\ell_i,q_i,q'_i,g_i\}_{i=1}^{c}$, where the $\ell_i$-s are
linear, the $q_i$-s and $q'_i$-s are quadratic and the $g_i$-s are
cubic such that $f = \sum_{j=1}^{c} \ell_j \cdot g_j +
\sum_{j=1}^{c} q_j \cdot q'_j $.
\end{introthm}

As mentioned above, prior to this result it was known that there
exist $C$ cubic polynomials $g_1,\ldots,g_C$ and a function $F$
such that $f=F(g_1,\ldots,g_C)$, where $C$ is a tower of height
$\exp(\poly(1/\delta))$ \cite{GreenTao07,KaufmanLovett08}. Thus,
our result greatly improves the dependance on $\delta$ and gives a
nice structure for the polynomial. We note that in their work
Green and Tao do show that such a nice structure exists when $d <
|\F|$ \cite{GreenTao07}, but no such result was known for smaller
fields (in addition $C$ needs to be even larger for such a nice
representation to hold).

Our third main result is for the case where $\deg(f)=4$ and
$\|f\|_{U^4} = \delta$. In such a case it is known
\cite{LMS08,GreenTao07} that we cannot hope to get a nice
structure as in Theorem~\ref{thm:intro:deg-3:gowers} as it may be
the case that $f$ has an exponentially small (in $n$) correlation
with all lower degree polynomials. However, we do manage to show
that there is some subspace $U \subset \F^n$ such that when
restricted to $V$, $f|_U$ is equal to some degree three
polynomial. Thus, $f$ does not have a correlation with a cubic
polynomial in the entire space but instead there is a large
subspace on which it is of degree three. In fact we show a more
general result. Namely, that there is a large subspace $V$, of
dimension $n-O(\log(1/\delta))$, that can be partitioned to
subspaces of dimension $n/\exp(\log^2(1/\delta))$ such that the
restriction of $f$ to any of the subspaces in the partition is of
degree three.

\begin{introthm}\label{thm:intro:deg-4:gowers}(quartic polynomials with high $U^4$ norm)
Let $\F$ be a finite field and $f \in \F[x_1,\ldots,x_n]$ a degree
four polynomial such that $\|f\|_{U^4}=\delta$. Then there exists
a partition of a subspace $V\subseteq \mathbb{F}^{n}$, of
dimension $\dim(V) \geq n-O(\log(1|/\delta))$, to subspaces
$\left\{ V_{\alpha}\right\} _{\alpha\in I}$, satisfying
$\dim(V_\alpha) =\Omega(n/|\F|^{\log^2(1/\delta)})$, such that for
every $\alpha\in I$, $f|_{V_{\alpha}}$ is a cubic polynomial.
\end{introthm}

\begin{remark}
Note that the structure guaranteed in
Theorem~\ref{thm:intro:deg-4:gowers} is shared by very few
polynomials. Specifically, a random polynomial of degree four is
unlikely to be equal to any degree three polynomial on any
subspace of dimension larger than, say, $n^{0.9}$. To see this
note that if $|\F|=p$ and $\dim(V)=d$ then there are roughly
$p^{d^3}$ cubic polynomials and $p^{d^4}$ quartic polynomials over
$V$. Furthermore, the map taking a quartic polynomial over $\F^n$
to its restriction is a linear map and so the fraction of quartic
polynomials that equal a degree three polynomial on $V$ is
(roughly) $p^{-d^4+d^3}$. As the total number of subspaces can be
bounded by $p^{n^2}$ we get that the fraction of quartic
polynomial that are equal to a degree three polynomial on some
subspace of dimension greater than $n^{0.9}$ is at most $p^{n^2 -
n^{3.6} + n^3} = \exp(-n^{3.6})$.
\end{remark}

This result has the same flavor as the inverse $U^3$ norm theorem
of \cite{GreenTao08}. There it was shown that if $f:\F_5^n \to
\F_5$ satisfies $\|f\|_{U^3}=\delta$ then there exists a subspace
$V$ of codimension $\poly(1/\delta)$, such that on an `average'
coset of $V$, $f$ is correlated with a quadratic polynomial.
Recently, Wolf \cite{Wolf09} proved a similar result for the case
of characteristic two, thus extending Samorodnitsky's argument
\cite{Samorodnitsky07}. The main difference between these results
and our result is that ours only holds for polynomials of degree
four whereas the results of
\cite{GreenTao08,Samorodnitsky07,Wolf09} hold for arbitrary
functions. On the other hand our result holds for the $U^4$ norm
compared to the $U^3$ norm studied there. Moreover, when
$\mathrm{char}(\F)>4$, using the same techniques we can actually
show that $f$ must have a structure similar to the one guaranteed
by Theorem~\ref{thm:intro:deg-4:bias}.

\begin{introthm}\label{thm:intro:deg-4:gowers:high
char} Let $\F$ be a finite field with $\mathrm{char}(\F)>4$ and $f
\in \F[x_1,\ldots,x_n]$ a degree four polynomial such that
$\|f\|_{U^4}=\delta$. Then
$$f = \sum_{i=1}^{R} \ell_i\cdot g_i + \sum_{i=1}^{r} q_i \cdot
q'_i\;,$$ for $r=O(\log^2(1/\delta))$  and
$R=\exp(\log^2(1/\delta))$ where $\ell_i$ is linear, $q_i,q'_i$
are quadratic and $g_i$ cubic.
\end{introthm}

\ignore{

If instead of knowing that the $U^4$ norm of $f$ is large we know
that its $U^3$ norm is large then we get a stronger conclusion,
namely that $f$ is constant on a large subspace.

\begin{introthm}\label{thm:intro:deg-4:gowers3}(quartic polynomials with high $U^3$ norm)
Let $\F$ be a finite field and $f \in \F[x_1,\ldots,x_n]$ a degree
four polynomial such that $\|f\|_{U^3}=\delta$. Then there exist a
constant $c = ***$ and a linear space $V \subset \F^n$ of
dimension $\dim(V) = n/c$ such that $f|_v = \alpha$ for some
$\alpha \in \F$.
\end{introthm}
}

\subsection{Proof Technique}\label{sec:intro:proof technique}

The main approach in all the proofs is to consider the space of
discrete partial derivatives of $f$ and look for some structure
there. We will explain the idea for the case of degree three
polynomials and then its extension to degree four polynomials.

Let $f$ be a degree three polynomials. Assume that $f$ has high
bias (alternatively, high $U^3$ norm). By a standard argument it
follows that a constant fraction of its derivatives, which are
degree $2$ polynomials, have high bias (high $U^2$ norm). By
Theorem~\ref{thm:Dickson} it follows that for a constant fraction
of the directions, the partial derivatives depends on a small
number of linear functions (same for the $U^3$ norm). Hence, in
the space of partial derivatives, a constant fraction of the
elements depend on a few linear functions. We now show that there
must be a small number of linear functions that `explain' this.
More accurately, we show that there exist a subspace $V \subset
\F^n$, of dimension $\dim(V) = n - O(1)$, and $O(1)$ linear
functions $\ell_1,\ldots,\ell_c$, such that for every $y \in V$ it
holds that $\Delta_y(f) = \sum_{i=1}^{c} \ell_i \cdot \ell^{(y)}_i
+ \ell^{(y)}_0$, where the $\ell_i^{(y)}$-s are linear functions
determined by $y$.

We are now basically done. Consider the subspace $U = \{x :
\ell_1(x)=\ldots=\ell_c(x)=0\}$. Then, for every $y\in V$ it holds
that $\Delta_y(f)|_U = \ell^{(y)}_0|_U$. This implies that $f|_V =
\sum_{i=1}^{c} \ell_i \cdot q_i + q_0$, where the $q_i$-s are
quadratic polynomials. As $\dim(V) = n-O(1)$ we obtain the same
structure (with a different constant $c$) for $f$.

To prove the result for biased degree four polynomials we follow
the footsteps of \cite{KaufmanLovett08} with two notable
differences. Let $f$ be such a polynomial. First, we pass to a
subspace on which all the partial derivatives of $f$ have low rank
as degree three polynomials. This steps relies on our results for
biased degree three polynomials. Then, as in
\cite{KaufmanLovett08}, we show that $f$ can be approximated by a
function of a few of its derivatives. Because of the properties of
the derivatives, this means that $f$ can be approximated well by a
function of a few quadratics and linear functions. We then show,
again following \cite{KaufmanLovett08}, that in such a case $f$ is
actually a function of a few quadratics and linear functions. Here
we heavily rely on properties of quadratic functions to avoid the
blow up in the number of polynomials approximating $f$ that occurs
in \cite{KaufmanLovett08,GreenTao07}. Finally, we show that if a
degree four polynomial is a function of several quadratic and
linear functions then it actually have a nice structure.

The proof for the case of degree four polynomials with high $U^4$
norm is more delicate. Assume that $f$ is such a polynomial. As
before, a constant fraction of the partial derivatives of $f$ are
degree three polynomials with high $U^3$ norm. By the result for
degree three polynomials we get that each of those partial
derivatives is of the form $\Delta_y(f) = \sum_{i=1}^{c}
\ell^{(y)}_i \cdot q^{(y)}_i + q^{(y)}_0$. Again we find a
subspace $V$, of constant co-dimension, such that for every $y \in
V$, $\Delta_y(f)$ has a nice structure. We now show that there
exist a small number of linear and quadratic functions
$\{\ell_i,q_i\}_{i=1}^{c}$ such that for every $y \in V$ it holds
that $\Delta_y(f) = \sum_{i=1}^{c} \ell_i \cdot q^{(y)}_i +
\sum_{i=1}^{c} \ell^{(y)}_i \cdot q_i + q^{(y)}_0$, where the
polynomials $\{\ell^{(y)}_i,q^{(y)}_i\}$ depend on $y$. This is
the technical heart of the proof. It now follows quite easily that
there is a subspace $U\subseteq V$ of dimension $n/\exp(c)$ such
that when restricted to $U$ all the functions $\{\ell_i,q_i\}$ are
fixed to constants. Thus, for every $y \in  U$ it holds that
$\deg(\Delta_y(f)) = 2$. So we get that $\deg(f|_{U}) = 3$. In
fact, by closely examining the argument above we show an even
stronger result. Namely, that we can partition a large subspace of
$\F^n$ to (affine) subspaces of dimension $n/\exp(c)$ such that on
each of the subspaces $f$ is equal to some cubic polynomial (that
may depend on the subspace).

\subsection{Organization}

In Section~\ref{sec:prelim} we give some basic definitions and
discuss properties of subadditive functions. In
Section~\ref{sec:degree 3} we prove the theorems concerning degree
three polynomials. In Section~\ref{sec:bias 4} we prove
Theorem~\ref{thm:intro:deg-4:bias} and in Section~\ref{sec:U4
norm} we prove Theorems~\ref{thm:intro:deg-4:gowers} and
\ref{thm:intro:deg-4:gowers:high char}.

\section{Preliminaries}\label{sec:prelim}

In this paper $\F$ will always be a prime field. We denote with
$\F_p$ the field with $p$ elements. As we will be considering
functions over $\F_p$ we will work modulo the polynomials $x_i^p -
x_i$. In particular, when we write $f=g$, for two polynomials, we
mean that they are equal as functions and not just as formal
expressions. This will be mainly relevant when we consider
quadratic polynomials (or higher degree polynomials) over $\F_2$.
More generally, we shall say that a function $f$ has degree $d$ if
there is a degree $d$ polynomial $g$ such that $f=g$. Note that
this does not have an affect on the bias and the Gowers norm.
Namely, the bias and $U^d$ norm of $f$ do not change when adding
multiplies of $x_i^p-x_i$. Finally we note that if all the partial
derivative of $f$ have degree at most $d-1$ then there is a
polynomial $g$ of degree at most $d$ such that $f=g$ (this is
easily proved by observing that a degree $k$ polynomial, all of
whose individual degrees are smaller than $|\F|$, always has a
partial derivative whose degree is $k-1$). From this point on we
shall use the notion of a function and a polynomial arbitrarily
without any real distinction.\\

The Fourier transform of a function $f:\F^n \to \F$ is defined as
$$\hat{f}(\alpha) = \E_{x \in \F^n}[f(x)\overline{\chi_\alpha(x)}]
\;,$$ where for $\alpha=(\alpha_1,\ldots,\alpha_n)$,
$\chi_\alpha(x) = \omega^{\sum_{i=1}^{n}\alpha_i x_i}$ where
$\omega = e^{\frac{2\pi i}{|\F|}}$ is a complex primitive root of
unity of order $|\F|$.
For more on Fourier transform see \cite{Stefankovic}.

We say that a function $h$ $\epsilon$-approximates a function $f$
if $\Pr_x[f(x) \neq h(x)]\leq \epsilon$.

\begin{definition}\label{def:gamma-approx}
Following \cite{KaufmanLovett08} we say that the distribution
induced by a set of functions $\{h_i\}_{i=1}^{m}$ (all from $\F^n$
to $\F$) is $\gamma$ close to the uniform distribution if for
every $\alpha_1,\ldots,\alpha_m \in \F$ it holds that $$\left|
\Pr_{x \in \F^n}[\forall 1\leq i \leq m,\; h_i(x) = \alpha_i] -
|\F|^{-m}\right| \leq \gamma|\F|^{-m} \, .$$
\end{definition}

The following well known lemma bounds the distance between
distributions using the Fourier transform.

\begin{lemma}\label{lem:XOR}
For $i =1\ldots m$ let $h_i:\F^n \to \F$ be a function. Then, the
distribution induced by the $h_i$-s is $\gamma$ close to uniform
if for every nontrivial linear combination
$h_\alpha=\sum_{i=1}^{m} \alpha_i h_i$,  we have that
$\bias(h_\alpha) \leq \gamma/|\F|^{3m/2}$.
\end{lemma}

\begin{proof}
Let $H:\F^n \to \F^m$ be defined as $H(x) =
(h_1(x),\ldots,h_m(x))$. For $y \in \F^m$ let $f(y)=\Pr_{x\in
\F^n}[H(x)=y]$. We have that
\begin{eqnarray*}
|\hat{f}(\alpha)| &=& \left|\E_{y\in
\F^m}\left[f(y)\overline{\chi_{\alpha}(y)}\right]\right| =
\left|\E_{y \in
\F^m}\left[\Pr_{x\in \F^n}[H(x)=y]\overline{\chi_{\alpha}(y)}\right]\right| \\
&=& |\F|^{-n-m}\left|\sum_{x \in \F^n}\chi_{\alpha}[H(x)]\right| =
|\F|^{-m}\bias\left(\sum_{i=1}^{m}\alpha_i h_i\right)\;.
\end{eqnarray*}
Therefore,
\begin{eqnarray*} \left(\sum_{y \in \F^m}\left|f(y) - |\F|^{-m}\right|\right)^{2}
&\leq& |\F|^m\sum_{y \in \F^m}\left|f(y) - |\F|^{-m}\right|^2
\\&=& |\F|^m\sum_{y \in \F^m} f(y)^2 -
2|\F|^{-m}f(y) +
|\F|^{-2m}\\
&=& \left(\sum_{\alpha \in \F^m} |\F|^{2m}\hat{f}(\alpha)^2
\right)- 1 =\left(\sum_{0\neq \alpha \in \F^m}
|\F|^{2m}\hat{f}(\alpha)^2 \right)
\\&=& \sum_{0\neq \alpha \in \F^m}
\bias\left(\sum_{i=1}^{m}\alpha_i h_i\right)^2 <
|\F|^{-2m}\gamma^2 \;.
\end{eqnarray*}
Hence, for every $y \in \F^m$ it holds that $|f(y) - |\F|^{-m}|<
|\F|^{-m}\gamma$, which is what we wanted to prove.
\end{proof}

\subsection{Subadditive functions}

As described in Section~\ref{sec:intro:proof technique} our proofs
are based on finding a structure for the space of partial
derivatives of the underlying polynomial $f$. For this end we need
a special case of the Bogolyubov-Chang lemma (see e.g.
\cite{Green}).

For a set $A \subseteq \F^n$ denote with $kA-kA$ the set $$kA-kA =
\{ a_1+\ldots+a_k - a_{k+1}-\ldots-a_{2k} \mid \forall i \; a_i
\in A\}\;.$$

\begin{lemma}[Bogolyubov-Chang]
\label{lem:B-C}  Let $A \subseteq U$ be a subset of a linear space
$U$ such that $|A|=\mu_0\cdot |U|$. Then, for some $k \leq
\max(1,\lceil\frac{1}{2}(\log_{\frac{|\F|}{|\F|-1/2}}(2/\mu_0)+2)\rceil)$,
$kA-kA$ contains a subspace $W$ of co-dimension at most
$\log_{\frac{|\F|-1/2}{|\F|-1}}(1/2\mu_0)$.
\end{lemma}

For completeness we give the proof here.

\begin{proof}
For $\mu \in (0,1)$ define  $\rho(\mu) = \frac{|\F| -
1/2}{|\F|}\cdot \mu$. We shall think of $A$ also as the
characteristic function of the set $A$ and denote with
$\{\hat{A}(\alpha)\}$ its fourier coefficients. Assume that there
is some $\alpha \neq 0$ such that $|\hat{A}(\alpha)| \geq
\rho(\mu_0)$. This means that there is some (affine) subspace $W$
of co-dimension at most one such that
$$|A\cap W|/|W| \geq \rho(\mu_0) \cdot |\F|/(|\F|-1) = \frac{|\F|
- 1/2}{|\F|-1}\cdot \mu_0 = (1+\epsilon) \mu_0\;,$$ where
$\epsilon =\frac{1}{2|\F|-2}$. In other words, the density of $A$
on $W$ is $(1+\epsilon)$ larger than its density over the entire
space. We continue restricting $A$ to co-dim one subspaces
(updating $\mu$ and considering $\rho(\mu)$ at each step) until
after at most $t=\log_{\frac{|\F|-1/2}{|\F|-1}}(1/2\mu_0)$ steps
we reach one of two possibilities. Either we get a subspace
$V\subseteq U$ of co-dimension at most $t$ such that $|A \cap V| >
|V|/2$, or $\widehat{A \cap V}(\alpha) < \rho(\mu)$ for every
$\alpha \neq 0$, where $\mu_0<\mu = |A \cap V|/|V|$. In the first
case it is clear that $(A \cap V) + (A \cap V) = V$ and so we
found a subspace $V$ of co-dimension at most $t$ contained in
$A+A$. In the second case where all the non-zero Fourier
coefficients are smaller than $\rho(\mu)$ we show that for
$k=\lceil
\frac{1}{2}(\log_{\frac{|\F|}{|\F|-1/2}}(2/\mu)+2)\rceil$ it holds
that $k(A \cap V) - k(A \cap V) = V$. For this end we follow the
proof of Lemma 4.4 in \cite{Green}. Let $B=A \cap V$. For $x \in
V$ denote with $r_k(x)$ the number of representations of $x$ as
$a_1+\ldots+a_k-a'_1-\ldots-a'_k$ where the $a_i$-s and $a'_i$-s
are from $B$. Clearly, $r_k(x)$ is equal to the sum, over all
$(y_1,\ldots,y_k,z_1,\ldots,z_{k-1})\in B^{2k-1}$, of $
A(y_1)\cdot A(y_2)\cdot\ldots\cdot A(y_k)\cdot
A(z_1)\cdot\ldots\cdot A(z_{k-1})\cdot A(y_1+\ldots+y_{k}-
z_1-\ldots-z_{k-1}-x)$. Writing the Fourier expansion $A$ and
using routine calculations we conclude that $$ r_k(x) =
|\F|^{(2k-1)n} \cdot
\sum_{\alpha}|\hat{B}(\alpha)|^{2k}\chi_\alpha(x) > |\F|^{(2k-1)n}
\cdot \left(\hat{B}(0)^{2k} - \sum_{\alpha\neq
0}|\hat{B}(\alpha)|^{2k}\right)\geq$$
$$|\F|^{(2k-1)n} \cdot \left(\hat{B}(0)^{2k} -
\rho(\mu)^{2k-2}\sum_{\alpha}|\hat{B}(\alpha)|^{2}\right) =
|\F|^{(2k-1)n} \cdot \left(\mu^{2k} - \rho(\mu)^{2k-2} \mu\right)
> 0\;,$$ where the last inequality follows from the choice of $k$
(we also used the fact that $A$ is a $0/1$ function). In
particular, $V \subseteq kA-kA$ as needed.
\end{proof}

We will mainly apply the lemma on sets $A \subseteq \F^n$
containing all directions where the partial derivatives of our
underlying polynomial $f$ are either very biased or have a high
Gowers norm. More generally we define the notion of a subadditive
function below.

\begin{definition}\label{def:sub additive}
Let $V \subset \F^n$ be a linear space. $\mathcal{F}:V \to \R^+$
is a subadditive function if for every $u,v \in V$ and $\alpha \in
\F$ it holds that $\mathcal{F}(\alpha\cdot u+v) \leq
\mathcal{F}(u)+\mathcal{F}(v)$.
\end{definition}

\begin{lemma}
\label{lem:subadditive} Let $\mathcal{F}:U\rightarrow\mathbb{R}^+$
be a subadditive function. Define, $A_{r}\triangleq\left\{ x\in
U\mid\mathcal{F}(x)\leq r\right\} $. If $|A_{r}|\geq \mu|U|$, then
there exists a vector space $V$ of co-dimension at most
$\log_{\frac{|\F|-1/2}{|\F|-1}}(1/2\mu) = O(\log(1/\mu))$ such
that for every $y \in V$ it holds that ${\cal F}(y) \leq 2r\cdot
\lceil\frac{1}{2}(\log_{\frac{|\F|}{|\F|-1/2}}(2/\mu)+2)\rceil+2r
= O(r \log(1/\mu))$.
\end{lemma}

\begin{proof}
The proof is immediate from Lemma~\ref{lem:B-C}. Let $V$ be the
subspace guaranteed by the lemma when applied on $A_r$. As $V
\subseteq kA_r-kA_r$, for $k \leq
\max(1,\lceil\frac{1}{2}(\log_{\frac{|\F|}{|\F|-1/2}}(2/\mu)+2)\rceil)$,
we get that ${\cal F}(y) \leq 2kr$ for every $y\in V$.
\end{proof}

A typical example of a subadditive function will be the rank of a
quadratic polynomial.

\begin{definition}\label{def:rank of quadratic}
Let $q$ be a degree two function over a prime field $\F$. We
define $\rank_2(q) = r$, where $r$ is the number of $\alpha_i$-s
that are non zero when considering the canonical representation of
$q$ in Theorem~\ref{thm:Dickson}.
\end{definition}

The following lemma is immediate.

\begin{lemma}\label{lem: rank is subadditive}
For two quadratic polynomials $q,q'$ and a constant $\alpha \in
\F$ we have that $\rank_2(q + \alpha q') \leq \rank_2(q) +
\rank_2(q')$.
\end{lemma}

A more interesting example is given in the following lemma.

\begin{lemma}\label{lem:subadditivity of derivatives}
Let $f$ be a cubic polynomial over a prime field $\F$. For every
$y \in \F^n$ define ${\cal F}(y) = \rank_2(\Delta_y(f))$. Then
$\cal F$ is a subadditive function.
\end{lemma}

\begin{proof}
The proof follows from the following simple observation
\begin{eqnarray*}
\Delta_y(f) + \Delta_z(f) &=& f(x+y)-f(x) + f(x+z)-f(x)\\ &=&
f(x+y+z) - f(x) - (f(x+y+z)-f(x+y) - (f(x+z)-f(x))) \\ &=&
\Delta_{y+z}(f)(x) - (\Delta_z(f)(x+y) - \Delta_z(f)(x))\\
&=& \Delta_{y+z}(f)(x) - \Delta_y\Delta_z(f)(x) \;.
\end{eqnarray*}
Indeed, we now get that ${\cal F}(y+z) = \rank_2(\Delta_{y+z}(f))
= \rank_2(\Delta_y(f) + \Delta_z(f) +\Delta_y\Delta_z(f)(x)) =
\rank_2(\Delta_y(f) + \Delta_z(f)) \leq \rank_2(\Delta_y(f)) +
\rank_2(\Delta_z(f))={\cal F}(y)+{\cal F}(z)$, where we used the
fact that adding a linear function to a quadratic polynomial does
not change its rank.
\end{proof}

\section{The structure of cubic polynomials}\label{sec:degree 3}

In this section we prove Theorems~\ref{thm:intro:deg-3:bias} and
\ref{thm:intro:deg-3:gowers}. As described in
Section~\ref{sec:intro:proof technique} both proofs are based on
finding a structure for the space of partial derivatives of $f$.

\subsection{Restricting the polynomial to a `good' subspace}\label{sec:3-good-subspace}

In this section we show that if a cubic $f$ ia biased or have a
large $U^3$ norm then there is a subspace $V \subseteq \F^n$ such
that for every $y \in V$ the rank of $\Delta_y(f)$ is relatively
small. We start by showing that if $f$ is biased or has a high
Gowers norm then so do many of its partial derivatives. The
following lemmas are well known and we prove them here for
completeness.

\begin{lemma}
\label{lem:biased der}Let
$f:\mathbb{F}_{p}^{n}\rightarrow\mathbb{F}_{p}$ be such that
$\bias(f) = \delta$. Then a fraction of at least
$\frac{1}{2}\delta^{2}$ of the partial derivatives $\Delta_{y}(f)$
satisfy $\bias(\Delta_y(f)) \geq \frac{1}{2}\delta^{2}$.
\end{lemma}
\begin{proof}
We first compute the expected bias of a partial derivative with
respect to a random direction. {\begin{eqnarray*}
\mathbb{E}_{y\in\mathbb{F}^{n}}\left[\mbox{bias}(\Delta_y(f))\right]
& = &
\mathbb{E}_{y\in\mathbb{F}^{n}}\left[\left|\mathbb{E}_{x\in\mathbb{F}^{n}}
\left[\omega^{\Delta_y(f)(x)}\right]\right|\right]
\geq\left|\mathbb{E}_{y\in\mathbb{F}^{n}}
\left[\mathbb{E}_{x\in\mathbb{F}^{n}}\left[\omega^{f(x+y)-f(x)}\right]\right]\right|\\
 & = & \left|\mathbb{E}_{y\in\mathbb{F}^{n},x\in\mathbb{F}^{n}}\left[\omega^{f(x+y)}\omega^{-f(x)}\right]\right|=
 \left|\mathbb{E}_{z\in\mathbb{F}^{n},x\in\mathbb{F}^{n}}\left[\omega^{f(z)}\omega^{-f(x)}\right]\right|\\
 & = & \left|\mathbb{E}_{z\in\mathbb{F}^{n}}\left[\omega^{f(z)}\right]
 \right|\left|\overline{\mathbb{E}_{x\in\mathbb{F}^{n}}\left[\omega^{f(x)}\right]}\right|=
 \delta\cdot\delta=\delta^{2} \;.\end{eqnarray*}
}

Therefore, by the fact that $\mbox{bias}(f)\leq1$, it follows that
$$\Pr_{y\in\mathbb{F}^{n}}\left[\mbox{bias}(\Delta_y(f))>\frac{1}{2}\delta^{2}\right]>\frac{1}{2}\delta^{2} \;.$$
\end{proof}

A similar result holds when $f$ has a high $U^d$ norm.

\begin{lemma}
\label{lem:gowers norm of der}Let
$f:\mathbb{F}_{p}^{n}\rightarrow\mathbb{F}_{p}$ be such that
$\|f\|_{U^d} = \delta$. Then a fraction of at least
$\frac{1}{2}\delta^{2^{d}}$ of the partial derivatives
$\Delta_{y}(f)$ satisfy $\|\Delta_y(f)\|_{U^{d-1}} \geq
\frac{1}{2}\delta^{2}$.
\end{lemma}

\begin{proof}
The proof is again immediate from the definition.
\begin{eqnarray*} \delta^{2^d} = \|f\|_{U^d}^{2^d} & = &
\left|\E_{x,y_1,\ldots,y_d}\left[\omega^{\Delta_{y_1}\ldots\Delta_{y_d}(f)(x)}\right]\right|\\
&\leq &
\E_{y_d}\left|\E_{x,y_1,\ldots,y_{d-1}}\left[\omega^{\Delta_{y_1}\ldots\Delta_{y_{d-1}}(\Delta_{y_d}(f))(x)}\right]\right|\\
& = & \E_{y}\left[\|\Delta_{y}(f)\|_{U^{d-1}}^{2^{d-1}}\right] \;.
\end{eqnarray*}
As before we get that
$$\Pr_{y\in\mathbb{F}^{n}}\left[\|\Delta_y(f)\|_{U^{d-1}}>\frac{1}{2}\delta^{2}\right]>\frac{1}{2}\delta^{2^{d}}
\;.$$
\end{proof}

We thus see that in both cases a constant fraction of all partial
derivatives of $f$ have high bias or high $U^2$ norm. From
Theorem~\ref{thm:Dickson} we get that if a partial derivative
(which is a quadratic function) has a high bias then it depends on
a few linear functions.

\begin{lemma}\label{lem:rank of biased quadratic}
Let $q$ be a quadratic polynomial over a prime field $\F$. Then
$q$ is a function of at most $\log_{|\F|}(\bias(q))+1$ linear
functions. More accurately, in the notations of
Theorem~\ref{thm:Dickson} the number of non zero $\alpha_i$-s is
at most $\log_{|\F|}(1/\bias(q))$.
\end{lemma}

\begin{proof}
See e.g. Lemmas 15-17 of \cite{BogdanovViola07}.
\end{proof}

The next lemma of Bogdanov and Viola \cite{BogdanovViola07} shows
that a similar result holds when a partial derivative has a high
$U^2$ norm.

\begin{lemma}\label{lem:rank of quadratic with high U2}(Lemma 15
of \cite{BogdanovViola07}) Every quadratic polynomial $q$ over a
prime field $\F$ is a function of at most
$\log_{|\F|}(1/\|q\|_{U^2})+ 1$ linear functions. Further, in the
notations of Theorem~\ref{thm:Dickson} the number of non zero
$\alpha_i$-s is at most $\log_{|\F|}(1/\|q\|_{U^2})$.
\end{lemma}

Concluding, we have proved the following lemma (recall
Definition~\ref{def:rank of quadratic}).

\begin{lemma}\label{lem:structur of non random cubic}
Let $f$ be a cubic polynomial.
\begin{enumerate}
\item If $\bias(f) = \delta$, then for (at least) a
$\frac{\delta^2}{2}$ fraction of $y \in \F^n$ it holds that
$\rank_2(\Delta_y(f)) \leq \log_{|\F|}(\frac{2}{\delta^2})$.

\item If $\|f\|_{U^2} = \delta$, then for (at least) a
$\frac{\delta^4}{2}$ fraction of $y \in \F^n$ it holds that
$\rank_2(\Delta_y(f)) \leq \log_{|\F|}(\frac{2}{\delta^2})$.

\end{enumerate}
\end{lemma}

We now combine Lemma~\ref{lem:subadditivity of derivatives} with
Lemma~\ref{lem:structur of non random cubic} and
Lemma~\ref{lem:subadditive} and obtain the following corollary.

\begin{corollary}\label{cor:good subspace of low rank}
Let $f$ be a cubic polynomial. If $\bias(f) = \delta$ or
$\|f\|_{U^2} = \delta$, then there exists a subspace $V \subseteq
\F^n$ such that $\dim(V) \geq n -  O(\log(\frac{1}{\delta}))$ and
such that for every $y \in V$ it holds that $\rank_2(\Delta_y(f))
= O(\log^2(\frac{1}{\delta}))$.
\end{corollary}

\subsection{The structure of low rank spaces}

So far we have established the existence of a subspace $V
\subseteq\F^{n}$ such that for every $y \in V$ it holds that
$\rank_{2}(\Delta_{y}(f)) = O(\log^{2}(\frac{1}{\delta}))$. We now
show that such spaces of low rank polynomials have a very
restricted structure. Namely, there exist
$r=O(\log^{2}(\frac{1}{\delta}))$ linear functions
$\ell_{1},\ldots,\ell_{r}$ such that every $\Delta_{y}(f)$ can be
written as $\Delta_{y}(f) = \sum_{i=1}^{r} \ell_{i}
\cdot\ell_{i}^{(y)} + \ell_{0}^{(y)}$, where the
$\ell_{i}^{(y)}$-s are linear functions determined by $y$. The
intuition behind this result is that $\rank_{2}(q+q')$ can be much
smaller than $\rank_{2}(q)+\rank_{2}(q')$ only if there is some
basis with respect to which $q$ and $q'$ share many linear
functions when represented in the form of
Theorem~\ref{thm:Dickson}. From this observation we deduce that if
we consider some function of maximal rank,
$q=\sum_{i=1}^{r}\ell_{i} \cdot\ell_{i}'$, and set $\left\{
\ell_{i},\ell'_{i}\right\} _{i=1}^{r}$ to zero (namely, consider
the subspace on which they all vanish), then on this subspace the
rank of the remaining quadratic functions decreases by a factor of
two. Repeating this argument we get that after setting at most
$4r$ linear functions to zero, all our quadratic functions become
linear functions.

\begin{lemma}
\label{lem:low rank space have few common functions} Let $M$ be a
linear space of quadratic functions satisfying $\rank_2(p) \leq r$
for all $p \in M$. Then there exists a subspace $V \subseteq\F^{n}
$ of co-dimension $\leq 4r$ such that $p|_{V}$ is a linear
function for all $p \in M$.
\end{lemma}

We shall give the proof for the case $\F=\F_2$. The proof for odd
characteristics is very similar (except that in the odd
characteristic case we have that the co-dimension of $V$ is $2r$
whereas in the even characteristic case it is $4r$).

\begin{proof}\sloppy
Let $g\in M$ be such that $\rank_2(g)=r$. By
Theorem~\ref{thm:Dickson}, $g$ can be expressed as
$g=\sum_{i=1}^{r}\ell_{2i-1} \cdot \ell_{2i} + \ell_0$. Denote
$V\triangleq \left\{
x\mid\ell_{1}(x)=\ell_{2}(x)=...=\ell_{2r}(x)=0\right\}$. We now
show that for every $h\in M$ it holds that $\rank_2(h|_V)\leq
\frac{r}{2}$. Repeating this argument we get that after setting at
most $2r + 2(r/2) + 2(r/4) + \ldots \leq 4r$ linear functions to
zero, the rank of all the quadratic functions in $M$ became zero.

Pick some $h\in M$ and denote $\rank_2(h|_V)=s$. As before, $h|_V$
can be expressed as $h|_{V}=\sum_{i=1}^{s}m_{2i-1}\cdot m_{2i} +
m_0$ (where the $m_{i}$-s are linear functions). Clearly the
functions $\{\ell_1,\ldots,\ell_{2r},m_1,\ldots,m_{2s}\}$ are
linearly independent. We can therefore write
$$h = \sum_{i=1}^{s}m_{2i-1}\cdot m_{2i} + m_0 +
\sum_{i=1}^{2r}\ell_i \cdot L_i \;,$$ where the $L_i$-s are linear
functions. Write $L_i = \tilde{m}_i + \tilde{\ell}_i +
\tilde{L}_i$ where $\tilde{m}_i \in
\mathrm{span}\{m_0,\ldots,m_{2s}\}$, $\tilde{\ell}_i \in
\mathrm{span}\{\ell_0,\ldots,\ell_{2r}\}$ and $\tilde{L}_i$ is
linearly independent of the $m_j$-s and $\ell_j$-s. Rearranging
terms we get that
$$h = \sum_{i=1}^{s}(m_{2i-1} + \ell'_{2i-1})\cdot (m_{2i}+ \ell'_{2i}) + (m_0 +\ell'_0) +
\tilde{h}(\ell_0,\ldots,\ell_{2r},\tilde{L}_1,\ldots,\tilde{L}_{2r})\;,$$
where each $\ell'_i$ is in the span of the $\ell_i$-s and
$\tilde{h}$ is a quadratic polynomial. Denote $m'_i = m_i +
\ell'_i$. It is clear that
$\ell_0,\ldots,\ell_{2r},\tilde{L}_1,\ldots,\tilde{L}_{2r}$ are
linearly independent of the $m'_i$-s (and vice versa).
Consequently,\footnote{From Theorem~\ref{thm:Dickson} it is clear
that for quadratic polynomials $q_1,q_2$ it holds that
$\rank_2(q_1(\bar{x})+q_2(\bar{y})) =
\rank_2(q_1(\bar{x}))+\rank_2(q_2(\bar{y}))$. }
$\rank_2(\sum_{i=1}^{s}m'_{2i-1}\cdot m'_{2i} + m'_0) +
\rank_2(\tilde{h}) = \rank_2(h) \leq r$. Hence,
$\rank_2(\tilde{h}) \leq r-s$. We now get that
$$r \geq \rank_2(g+h) = \rank_2\left(\sum_{i=1}^{s}m'_{2i-1}\cdot m'_{2i} + m'_0
+\tilde{h}(\ell_0,\ldots,\ell_{2r},\tilde{L}_1,\ldots,\tilde{L}_{2r})
+g\right) =$$  $$\rank_2\left(\sum_{i=1}^{s}m'_{2i-1}\cdot m'_{2i}
+ m'_0\right) + \rank_2\left(g +
\tilde{h}(\ell_0,\ldots,\ell_{2r},\tilde{L}_1,\ldots,\tilde{L}_{2r})\right)
\geq s + (r - (r-s)) =2s \;,$$ where we used the fact that
$\rank_2(g+\tilde{h})\geq \rank_2(g)-\rank_2(\tilde{h})$. As we
showed that $r \geq 2s$ the proof is completed.
\end{proof}

\ignore{

Repeating this argument $r$ times yield the next
corollary.\footnote{In fact, the proof above and the corollary can
be made tighter but as it will not improve our end result we did
not make an attempt to improve the parameters.}

\begin{corollary}
\label{cor:Low rank few common functions}Let $M$ be a vector space
of quadratic functions satisfying $\rank_2(M)\leq r$. Then there
exists a set $A=\left\{ \ell_{i}\right\} _{i=1}^{r^{2}}$ of
linearly independent linear functions such that
$\rank_2([M]_A)=0$. In other words, every $q \in M$ can be written
as $q=\sum_{i=1}^{r^{2}}\ell_{i}\ell_{i}'+\ell_{0}'$ for some
linear functions $\left\{ \ell_{i}'\right\} _{i=0}^{r^{2}}$.
\end{corollary}

}

\subsection{Completing the proofs}

We are no ready to complete the proofs of
Theorems~\ref{thm:intro:deg-3:bias} and
\ref{thm:intro:deg-3:gowers}.

\begin{proof}[Proof of Theorem~\ref{thm:intro:deg-3:gowers}]
By Corollary~\ref{cor:good subspace of low rank} we get that if
$\|f\|_{U^2} = \delta$, then there exists a subspace $V \subseteq
\F^n$ such that $\dim(V) \geq n - O(\log(1/\delta))$ and such that
for every $y \in V$ it holds that $\rank_2(\Delta_y(f)) =
O(\log^2(\frac{1}{\delta}))$. Lemma~\ref{lem:low rank space have
few common functions} implies that there are at most
$r=O(\log^2(\frac{1}{\delta}))$ linear functions
$\ell_1,\ldots,\ell_r$ such that for every $y \in V$ we have that
$\Delta_y(f) = \sum_{i=1}^{r}\ell_i \cdot \ell^{(y)}_i +
\ell^{(y)}_0$. Let $U = \{x\in V \mid\ell_1(x) = \ldots \ell_r(x)
= 0 \}$. Then $U$ is a linear space of dimension $\dim(U) \geq
n-O(\log^2(\frac{1}{\delta}))$. For every $y \in U$ we have that
$\Delta_y(f)|_U = \ell^{(y)}_0|_U$. Hence, for every $y \in U$,
$\deg(\Delta_y(f)) \leq 1$. Therefore, $\deg(f|_U) \leq 2$. Let
$\ell'_1,\ldots,\ell'_t$ be linearly independent linear functions
such that $x \in U$ iff $\ell'_1(x)=\ldots=\ell'_t(x)=0$. It
follows that we can write $f = \sum_{i=1}^{t}\ell'_i \cdot q_i +
q_0$ for some quadratic polynomials $\{q_i\}$. As $t = n - \dim(U)
 = O(\log^2(\frac{1}{\delta}))$ the result follows.
\end{proof}

The proof of Theorem~\ref{thm:intro:deg-3:bias} is essentially the
same except that we make another small optimization that reduces
the required number of quadratic functions.

\begin{proof}[Proof of Theorem~\ref{thm:intro:deg-3:bias}]
By the same argument as above we get that $f =
\sum_{i=1}^{t}\ell_i \cdot q_i + q_0$ for some quadratic
polynomials $\{q_i\}$ and linear functions $\{\ell_i\}$, where $t
= O(\log^2(\frac{1}{\delta}))$. For convenience we shall assume
w.l.o.g. that
\begin{equation}\label{eq: f before reducing t} f = \sum_{i=1}^{t}x_i \cdot q_i + q_0
\;.\end{equation} The following lemma shows that by adding a few
more linear functions we can assume that no nontrivial linear
combination of the $q_i$-s has a low rank.

\begin{lemma}\label{lem:regularizing a set of quadratics}
Let $q_1,\ldots,q_t$ be quadratic polynomials over $\F^n$. Then,
for every $r$ there exist a subspace $V \subset \F^n$ of dimension
$\dim(V) \geq n - t(r+1)$, and $t' \leq t$ indices
$i_1,\ldots,i_{t'}$ such that for every affine shift $V'$ of $V$
the following holds
\begin{enumerate}
\item For all $i$, $q_i|_{V'} \in
\mathrm{span}\{1,q_{i_1}|_{V'},\ldots,q_{i_{t'}}|_{V'}\}$.

\item For any non trivial linear combination we have that
$\rank_2\left(\sum_{j=1}^{t'}\alpha_i q_{i_j}|_{V'}\right) >r$.
\end{enumerate}
\end{lemma}

\begin{proof}
The proof is by induction on $t$. For $t=1$ the claim is clear: If
$\rank_2(q_1) >r$ then we are done. Otherwise we have $q_1 =
\sum_{i=1}^{r}\ell_{2i-1}\ell_{2i} + \ell_0$. Letting $V = \{x\mid
\ell_0(x)=\ell_2(x)=\ldots=\ell_{2r}(x)=0\}$ the claim follows
(indeed notice that passing to an affine shift of $V$ simply means
fixing the $\ell_i$-s to arbitrary values). Assume now that we
have $q_1,\ldots,q_t$ and that (w.l.o.g.) $\rank_2\left(q_t +
\sum_{i=1}^{t-1}\alpha_i q_i\right) \leq r$. Write $q_t +
\sum_{i=1}^{t-1}\alpha_i q_i = \sum_{i=1}^{r}\ell_{2i-1}\ell_{2i}
+ \ell_0$. Set $V = \{x\mid
\ell_0(x)=\ell_2(x)=\ldots=\ell_{2r}(x)=0\}$. Then, $q_t|_V \in
\mathrm{span}\{q_1|_V.\ldots,q_{t-1}|_V\}$. As $\dim(V) = n-(r+1)$
the claim follows by applying the induction argument to
$q_1|_V.\ldots,q_{t-1}|_V$ (again the claim about any affine shift
follows easily).
\end{proof}

We continue with the proof of the theorem. Having
Equation~\eqref{eq: f before reducing t} in mind we set $U =
\{(0,\ldots,0,x_{t+1},\ldots,x_n)\} \subset \F^n$. Applying
Lemma~\ref{lem:regularizing a set of quadratics} on
$q_1|_U,\ldots,q_t|_U$ with $r=\log_{|\F|}(2/\delta)$ we get that
there is a subspace $W \subset U$ and $t' \leq t$ such that:
$\dim(W) \geq \dim(U) - (r+1)t \geq n - (r+2)t =
n-O(\log^3(\frac{1}{\delta}))$; w.l.o.g. for every $i=1\dots t$,
$q_i|_W \in \mathrm{span}\{q_1|_W,\ldots,q_{t'}|_W\}$; any
nontrivial linear combination of $q_1|_W,\ldots,q_{t'}|_W$ has
rank larger than $r$. By applying an invertible linear
transformation\footnote{This step is not really required but we
continue using it just to make the proofs easier to read.} we can
further assume that $W = \{x \in \F^m \mid x_1=\ldots=x_m=0\}$ for
some $m \leq (r+2)t$. For $i=1\ldots t'$ let $q'_i = q_i|_W$. Note
that $q'_i$ does not contain any of the variables
$x_1,\ldots,x_m$. We can rewrite Equation~\eqref{eq: f before
reducing t} as\footnote{We will later explain why $q_0$
`disappeared' from this expression.}
\begin{equation}\label{eq:
f after reducing t} f = \sum_{i=1}^{t'}\ell'_i q'_i +
\sum_{i=1}^{t}\sum_{j=1}^{m}x_i x_j \ell_{i,j}\;,\end{equation}
where the $\ell'_i$-s are linearly independent linear functions in
$x_1,\ldots,x_{t}$. We now show that $t' < \log_{|\F|}(2/\delta)$.
Assume for contradiction that $t' \geq \log_{|\F|}(2/\delta)$.  As
$$\bias(f) =
\E_{\alpha_1,\ldots,\alpha_{t'}}\left[\bias(f(x_{1},\ldots,
x_n)|_{(\ell'_1,\ldots,\ell'_{t'})=(\alpha_1,\ldots,\alpha_{t'})}\right]\;,$$
\sloppy there exists an assignment
$(x_1,\ldots,x_m)=(\beta_1,\ldots,\beta_m)$ satisfying
$(\ell'_1,\ldots,\ell_{t'}) = (\alpha_1,\ldots,\alpha_{t'})\neq 0$
such that
$$\bias\left(\sum_{i=1}^{t'}\alpha_i q'_i  +
\sum_{i=1}^{t}\beta_i\sum_{j=1}^{m} \beta_j \ell_{i,j} \right)
\geq \delta - \frac{1}{|\F|^{t'}} \geq \delta/2\;.$$ Therefore,
for some constants $\alpha_1,\ldots,\alpha_{t'}$ (where not all
$\alpha_1,\ldots,\alpha_{t'}$ are zero) we have that
$$\bias\left( \sum_{i=1}^{t'}\alpha_i q'_i + \ell\right)\geq \delta/2 \;,$$ for some
linear function $\ell$. By Lemma~\ref{lem:rank of biased
quadratic} we get that
$$\rank_2\left(\sum_{i=1}^{t'}\alpha_i q'_i\right) =
\rank_2\left(\sum_{i=1}^{t'}\alpha_i q'_i + \ell\right) \leq
\log_{|\F|}(1/(\delta/2))  = r \;,$$ in contradiction to the
choice of $q'_1,\ldots,q'_{t'}$.

To complete the proof we explain the reason for dropping $q_0$.
Indeed, consider Equation~\eqref{eq: f before reducing t}. Let $U
= \{x\mid x_1=\ldots=x_t=0\}$. Set $\tilde{q}_i = q_i|_U$. Then we
can rewrite \eqref{eq: f before reducing t} as
$\sum_{i=1}^{t}x_i\tilde{q}_i + \tilde{q}_0 + \sum_{i=1}^{t}x_i
\sum_{j=1}^{t} x_{j} \ell_{i,j}$, for some linear functions
$\ell_{i,j}$. Now, for some $\alpha_1,\ldots,\alpha_t$ we get that
$\bias(\sum_{i=1}^{t}\alpha_i\tilde{q}_i + \tilde{q}_0 +
\sum_{i=1}^{t}\alpha_i \sum_{j=1}^{t} \alpha_j \ell_{i,j}) \geq
\delta.$ Lemma~\ref{lem:rank of biased quadratic} implies that
$\rank_2(\sum_{i=1}^{t}\alpha_i\tilde{q}_i + \tilde{q}_0) \leq
\log_{|\F|}(1/\delta)$ and so we can replace $\tilde{q}_0$ by a
linear combination of the other $\tilde{q}_i$-s and a function
depending on a few linear functions. By passing to a (possibly
affine) subspace of dimension at least $n -
\log_{|\F|}(1/\delta)-1$ we get a representation for $f$ without
$q_0$. This operation increases $t'$ in Equation~\eqref{eq: f
after reducing t} by no more than $\log_{|\F|}(1/\delta)+1$ and so
we are done.
\end{proof}

\section{The structure of biased 4 degree polynomials}\label{sec:bias 4}

In this section we prove Theorem~\ref{thm:intro:deg-4:bias} on the
structure of biased degree 4 polynomials. As in the case of cubic
polynomials, we shall focus our attention on a subspace on which
all of derivatives have a small rank (a cubic polynomial is of low
rank if it depends on a small number of linear and quadratic
functions). By a lemma of Bogdanov and Viola
\cite{BogdanovViola07} (Lemma \ref{lem:Bogdanov and Viola}) we get
that $f$ can be well approximated by a function of a small number
of its derivatives (which in our case, are all of low rank). Thus,
$f$ is well approximated by a function of a few linear and
quadratic polynomials. By passing to a subspace we can assume that
$f$ is well approximated by a function of a small number of
quadratic polynomials. Lemma~\ref{lem:regularizing a set of
quadratics} implies that  (possibly on a slightly smaller
subspace) $f$ can be well approximated by a function of a small
number of quadratics, that every nontrivial linear combination of
them has a high rank. We then show that in this case those
quadratic functions are in fact {\em strongly regular} (a notion
that we later explain) and therefore by a theorem of Kaufman and
Lovett \cite{KaufmanLovett08}, $f$ in fact equals a function in
those quadratic (on the subspace). We then finish the proof by
showing that in this case $f$ also have a nice structure.

\ignore{
 and deduce, using lemma \ref{lem:rank implie
regular}, that those family of function is strong regular. Using
lemma \ref{lem:Kaufman Lovett} of Kaufman and Lovett (\cite{KL08})
we deduce that our degree 4 polynomial can compute precisely as a
function $F$ of those quadratic functions. We'll conclude by argue
that $F$ must have a degree 2, otherwise our original polynomial's
degree will not be two. In conclusion, after reconsider the linear
functions set to zero, every biased degree 4 polynomial have the
form $\sum_{i=1}^{q}l_{i}T_{i}+\sum_{i=1}^{r}Q_{i}Q_{i}'+Q_{0}$. }

\subsection{Restricting the polynomial to a `good'
subspace}\label{sec:bias 4: restricting}

In this section we prove an analogous result to
Corollary~\ref{cor:good subspace of low rank}. We first define the
rank of a cubic polynomial.

\begin{definition}\label{def:cubic rank}
Let $g$ be a degree three polynomial. We define $\rank_3(g)$ to be
the minimal integer $r$ for which there are $r$ linear functions
$\ell_1,\ldots,\ell_r$ and $r+1$ quadratic functions
$q_0,\ldots,q_r$ such that $g=\sum_{i=1}^{r}\ell_{i}q_{i}+q_{0}$.
\end{definition}

\begin{lemma}\label{lem:good subspace of low rank for quartics}
Let $f$ be a degree four polynomial satisfying $\bias(f)=\delta$.
Then there exist a linear subspace $V \subseteq \F^n$ of dimension
$\dim(V) \geq n -O(\log_{|\F|}(1/\delta))$, such that for every $y
\in V$ $\rank_3(\Delta_y(f)) = \log^{O(1)}(1/\delta)$.
\end{lemma}

\begin{proof}
As before, define ${\cal F}(y) \eqdef \rank_3(\Delta_y(f))$. It is
again not difficult to see that $\cal F$ is a subadditive
function. By Lemma~\ref{lem:biased der} we get that there is a
subset $S \subseteq \F^n$ of size $\frac{\delta^2}{2}\cdot \F^n$
such that for all $y \in S$, $\bias(\Delta_y(f)) \geq
\frac{\delta^2}{2}$. Theorem~\ref{thm:intro:deg-3:bias} implies
that for every $y \in S$ it holds that $\rank_3(\Delta_y(f)) =
O(\log^4(\frac{1}{\delta}))$. From Lemma~\ref{lem:subadditive} it
follows that there is a linear subspace $V \subseteq \F^n$ with
$\dim(V) \geq n - O(\log_{|\F|}(1/\delta))$, such that for every
$y \in V$ $\rank_3(\Delta_y(f)) = O(\log^5(\frac{1}{\delta}))$.
\end{proof}
\sloppy By applying an invertible linear transformation we can
assume that $V = \{x :x_1=\ldots=x_{m}=0\}$ for some $m=
O(\log_{|\F|}(1/\delta))$. We now have \begin{equation}\label{eq:
quartic on subspace} f = \sum_{i=1}^{m}x_i g_i + f' \;,
\end{equation}
where $f'=f'(x_{m+1},\ldots,x_n)$. Moreover, by
Lemma~\ref{lem:good subspace of low rank for quartics} it follows
that for every $y = (0,\ldots,0,y_{m+1},\ldots,y_n)$,
$\rank_3(\Delta_y(f)) = O(\log^5(\frac{1}{\delta}))$. Notice that
for every such $y$ it holds that $$\Delta_y(f) = \sum_{i=1}^m x_i
\Delta_y(g_i) + \Delta_y(f')\;.$$ Hence, $\rank_3(\Delta_y(f'))
\leq \rank_3(\Delta_y(f)) + m$. We now fix some value to
$x_1,\ldots,x_m$, such that
$\bias(f(\alpha_1,\ldots,\alpha_m,x_{m+1},\ldots,x_n)) \geq
\delta$. Let
\begin{equation}\label{eq:quartic setting to constants}\tilde{f}(x_{m+1},\ldots,x_n) \eqdef
f(\alpha_1,\ldots,\alpha_m,x_{m+1},\ldots,x_n)\;.\end{equation} It
follows that $\bias(\tilde{f}) \geq \delta$ and that for every
$y=(y_{m+1},\ldots,y_n)$, $\rank_3(\Delta_y(\tilde f))
=\rank_3(\Delta_y(f'))= O(\log^5(\frac{1}{\delta}))$ (note that
$\deg\left(\Delta_y(\tilde{f})-\Delta_y(f')\right)=2$ so they have
the same rank). From now on we will only consider $\tilde{f}$ and
not $f$. Observe that if we prove
Theorem~\ref{thm:intro:deg-4:bias} for $\tilde{f}$ then by
considering Equations~\eqref{eq: quartic on subspace} and
\eqref{eq:quartic setting to constants} we get the required result
for $f$ itself.

\ignore{ Therefore, we abuse notations and assume, from now on,
w.l.o.g., that $\tilde{f}=f$ and $m=0$. }

\subsection{Computing $\tilde f$ using a few quadratics}

We now show that there is a large subspace on which $f$ can be
approximated by a function of a few quadratic polynomials. The
following lemma of Bogdanov and Viola shows that if $f$ is biased
then it can be well approximated by a small set of partial
derivatives.

\begin{lemma}
\label{lem:Bogdanov and Viola}(Lemma 24 from
\cite{BogdanovViola07}) Let $f:\mathbb{F}^{n}\to\mathbb{F}$ be a
function over a finite field $\F$ with $\bias(f)=\delta$. Then
there are $t$ directions $a_{1},...,a_{t}$ and a function $H$ such
that $H(\Delta_{a_1}(f),\ldots,\Delta_{a_t}(f))$
$\epsilon$-approximates $f$, where $t\leq
(1+\log\frac{1}{\epsilon})\left(|\mathbb{F}|/\delta\right)^{O(1)}$.
\end{lemma}

By the construction of $\tilde f$ we know that each of its partial
derivatives is of rank $\log^{O(1)}(1/\delta)$ and that
$\bias(\tilde{f})\geq \delta$. Thus, Lemma~\ref{lem:Bogdanov and
Viola} guarantees that $\tilde f$ can be well approximated using a
few quadratics.

\begin{corollary}\label{cor:tilde f approx by quadratics}
For every $\epsilon >0$ there are
$c=(1+\log\frac{1}{\epsilon})\left(|\mathbb{F}|/\delta\right)^{O(1)}$
quadratic polynomials $Q_1,\ldots,Q_c$ and a function $H$ such
that $\tilde f$ is $\epsilon$-approximated by $H(Q_1,\ldots,Q_c)$.
\end{corollary}

The next lemma, which is the main lemma of \cite{KaufmanLovett08}
shows that if the approximation is good enough (i.e. $\epsilon$ is
small), and if the quadratics satisfy the {\em strong regularity}
property then $\tilde f$ can in fact be computed by a small number
of quadratics.

\begin{definition}\label{def:strongly regular}
(strongly regular quadratic functions) We say that a family of
quadratic functions $\left\{ Q_{i}\right\} _{i=1}^{m}$ is
\emph{$\gamma$- strongly regular} if the following holds for every
$x_0 \in \F^n$: for independent uniform random variables
$Y_{1},...,Y_{5}$ the joint distribution of
$$\left\{Q_{j}\left(x_0+\sum_{i\in I}Y_{i}\right) \mid
j\in[m],I\subseteq[5],1\leq|I|\leq2\right\}$$ is $\gamma$ close to
the uniform distribution (recall
Definition~\ref{def:gamma-approx}).
\end{definition}

This definition is a restricted version of Definition 8 of
\cite{KaufmanLovett08} for quadratic polynomials. The interested
reader is referred to that paper for the general definition for
higher degree polynomials.

\begin{lemma}
\label{lem:Kaufman Lovett}(Lemma 13 from \cite{KaufmanLovett08})
Let $f(x)$ be a degree $d$ polynomials, $h_{1},...,h_{m}$
polynomials of degree less than $d$ and
$H:\mathbb{F}^{m}\to\mathbb{F}$ a function such that
\begin{itemize}
\item $H(h_{1},...,h_{m})$ $\epsilon$-approximates $f$ where
$\epsilon\leq2^{-2(d+1)}$.

\item $\left\{ h_{i}\right\} _{i=1}^{m}$ is a $\gamma$-strongly
regular family where $\gamma\leq\min\left\{
2^{-d},2^{-m}\right\}$.
\end{itemize}
Then there exists a function $F:\mathbb{F}^{m}\to\mathbb{F}$ such
that $f=F(h_{1},...,h_{m})$.
\end{lemma}

In other words, the lemma says that if $f$ is well approximated by
a family of strongly regular functions then it can actually be
computed everywhere by the functions in the family. We shall now
show that if $q_1,\ldots,q_c$ are quadratic polynomials such that
the rank of every nontrivial linear combination of them is high,
then they are strongly regular. This will imply (by
Corollary~\ref{cor:tilde f approx by quadratics}) that $\tilde f$
is a function of a few quadratics and therefore so is $f$.

\begin{lemma}
\label{lem:rank implies regular}Let $\left\{ Q_{i}\right\}
_{i=1}^{m}$ be a family of quadratic functions such that for every
nontrivial linear combination,
$\rank_2(\sum_{i=1}^{m}\alpha_{i}Q_{i})\geq R$. Then $\left\{
Q_{i}\right\}_{i=1}^{m}$is a $\gamma$-strongly regular family for
$\gamma=|\F|^{3m/2-R/4}$.
\end{lemma}

\begin{proof}
The proof is based on the analogy between quadratic functions and
matrices.

\begin{definition}\label{def:matrix}
Let $Q:\mathbb{F}^{n}\to\mathbb{F}$ be a quadratic polynomial and
$A\in\mathbb{F}^{n\times n}$ an $n\times n$ matrix. We say that
$A$ represents $Q$ if there exists a linear function $\ell$ such
that $Q(x)=x^{t}Ax+\ell(x)$.
\end{definition}

Notice that there may be many different matrices representing the
same polynomial $Q$. For example, every antisymmetric matrix
represents the zero function. More generally, if $S$ is
antisymmetric then $A$ and $A+S$ represent the same polynomial.

\begin{lemma}
Let $q$ be a quadratic polynomial. Then $\rank_2(q)$ (recall
Definition~\ref{def:rank of quadratic}) is equal to the minimal
rank of a matrix representing $q$. Moreover, for every matrix $A$
representing $q$ we have that $\rank(A+A^t)/2\leq \rank_2(q) \leq
\rank(A+A^t)$.
\end{lemma}

We shall prove the lemma for $\F=\F_2$. The proof for other fields
is similar.

\begin{proof}
Let $\rank_2(q)=r$. Then $q$ can be expressed as
$\sum_{i=1}^{r}\left(\sum_{j=1}^{n}a_{i,j}x_{j}\right)
\left(\sum_{j=1}^{n}b_{i,j}x_{j}\right) + \ell(x)$. Set
$A=(a_{i,j}),B=(b_{i,j})\in \F^{r \times n}$. It is clear that
$A^{t}B$ represents $q$ and that $\rank(A^t B) \leq r$. On the
other hand, if $q$ can be represented by a rank $r$ matrix $A$,
then let $\ell_1,\ldots,\ell_r$ be a basis for the rows of $A$,
when interpreted as linear functions.\footnote{I.e.
$(a_1,\ldots,a_n) \leftrightarrow \sum_{i=1}^{n}a_i \cdot x_i$.}
Let $A_i$ be the $i$-th row of $A$ and denote $A_i =
\sum_{j=1}^{r} \alpha_{i,j} \ell_j$. We have that for some linear
function $\ell$, \[ q-\ell=x^{t}Ax=\sum_{i=1}^{n}x_{i}A_{i}(x)
=\sum_{i=1}^{n}x_{i} \sum_{j=1}^{r}\alpha_{i,j}\ell_{j}
=\sum_{j=1}^{r}\ell_{j}\left(\sum_{i=1}^{n}\alpha_{i,j}x_{i}\right)=
\sum_{j=1}^{r}\ell_{j}\ell'_{j}\;,\] where
$\ell'_1,\ldots,\ell'_r$ are linear functions. This implies that
$\rank_2(q)\leq r$. Thus, $\rank_2(q)=\min\left\{ \rank(A)\mid
q(x)=x^{t}Ax+\ell(x)\right\} $.

To prove the second claim, let $A$ be any matrix representing $q$.
We first change the basis of the space so that with respect to the
new basis $q$ will have the form of Theorem~\ref{thm:Dickson}. Let
$T$ be an invertible matrix representing the change of basis.
Clearly, $T^t A T$ represents $q \circ T = \sum_{i=1}^{r}
x_{2i-1}x_{2i} + \ell$, where $r = \rank_2(q)$. Thus, the matrix
$T^t A T$ can be written as $D+S$ where $D$ is a block diagonal
matrix consisting of $r$ nonzero blocks of size $2 \times 2$ and
$S$ is a symmetric matrix. We also note that for each $2 \times 2$
diagonal block $C$ of $D$ it holds that $C+C^t \neq 0$. We thus
get that
$$\rank(A+A^t) = \rank(T^t(A+A^t)T)= \rank(D+S + D^t + S^t) =
\rank(D+D^t)\;.$$ Now, for every $2 \times 2$ diagonal block $C$
of $D$ we have that $1\leq \rank(C+C^t) \leq 2$ and so
$$\rank_2(q) =r \leq \rank(D+D^t) \leq 2r = 2\rank_2(q)\;.$$
This completes the proof of the Lemma.\footnote{From the proof it
actually follows that over $\F_2$, $\rank_2(q)=\rank(A+A^t)/2$ but
this is not the case for other prime fields.}
\end{proof}

We continue the proof of Lemma~\ref{lem:rank implies regular}.
Using the above observation we now prove that any nontrivial
linear combination
$\sum_{k\in[m],I\subseteq[5],1\leq|I|\leq2}\alpha_{k,I}Q_{j}(x+\sum_{i\in
I}Y_{i})$ has high rank (as a quadratic polynomial in the
variables $Y_1\cup\ldots\cup Y_5$).

Fix $x=x_0$ and let $A_{k}$ be a matrix representing $Q_{k}$.
Notice that the quadratic polynomial $Q_{k}(x_0+\sum_{i\in
I}Y_{i})$ (in the variables $\cup_{i=1}^{5}Y_i$) can be
represented by a block matrix $B^{k,I}\in\mathbb{F}^{5n\times5n}$.
Indeed, consider a $5\times 5$ matrix that has $1$ in the
$(i,j)$-position iff $i,j\in I$, and zeros otherwise. Now, replace
any $1$ by the matrix $A_k$ and every $0$ by the $n \times n$ zero
matrix. It is an easy calculation to see that this matrix
represents $Q_{k}(x_0+\sum_{i\in I}Y_{i})$. We shall abuse
notations and for $i,j \in I$ say that $(B^{k,I})_{i,j} = A_k$,
and that otherwise $(B^{k,I})_{i,j} = 0$.

Clearly, the linear combination $$Q'\eqdef
\sum\{\alpha_{k,I}Q_{k}(x+\sum_{i\in I}Y_{i}) \mid
{k\in[m],I\subseteq[5],1\leq|I|\leq2}\}$$ is represented by the
matrix $$C\triangleq\sum\{\alpha_{k,I}B^{k,I} \mid
{k\in[m],I\subseteq[5],1\leq|I|\leq2}\}\;.$$ Observe that for
$i\neq j\in[5]$, $C_{i,j}=\sum_{k\in[m]}\alpha_{k,\left\{
i,j\right\} }A_{k}$. We now show that if for some $i\neq j\in[5]$
and $k\in[m]$ it holds that $\alpha_{k,\left\{ i,j\right\} }\neq0$
then the rank of $C^t+C$ (and hence of $Q'$) is high.
\begin{eqnarray*} \rank_2(Q') &=&
\rank_2\left(\sum\left\{\alpha_{k,I}Q_{k}(x+\sum_{i\in I}Y_{i})
\mid {k\in[m],I\subseteq[5],1\leq|I|\leq 2}\right\}\right)\\
&\geq& \frac{1}{2}\rank(C+C^{t}) \geq
\frac{1}{2}\rank(C_{i,j}+C_{j,i}^{t})\\
&=& \frac{1}{2}\rank\left(\sum_{k\in[m]}\alpha_{k,\left\{
i,j\right\}}\left(A_{k}+A_{k}^{t}\right)\right)\\
&\geq& \frac{1}{4}\rank_2\left(\sum_{k\in[m]}\alpha_{k,\left\{
i,j\right\} }Q_{k}\right)>\frac{1}{4}R \;.
\end{eqnarray*}
If it is not the case, namely, for all $i\neq j\in[5],k\in[m]$
$\alpha_{k,\left\{ i,j\right\} }=0$, then there is some $i\in[5]$
and $k\in[m]$ such that $\alpha_{k,\left\{ i\right\} }\neq0$ and
we get that same result by considering $C_{i,i}$ instead.

To conclude, every nontrivial linear combination of $\left\{
Q_{j}(x+\sum_{i\in I}Y_{i})\right\}
_{k\in[m],I\subseteq[5],1\leq|I|\leq2}$ has rank grater than
$\frac{1}{4}R$. Lemma~\ref{lem:rank of biased quadratic} implies
that the bias of every such linear combination is bounded by
$|\F|^{-R/4}$. It now follows by Lemma~\ref{lem:XOR} that the
distribution is $|\F|^{3m/2-R/4}$ close to the uniform
distribution as needed.
\end{proof}

We thus get the following corollary.

\begin{corollary}\label{cor:approximating by regular}
Let $g(x)$ be a degree $d$ polynomials, $q_{1},...,q_{m}$
quadratic polynomials and $H:\mathbb{F}^{m}\to\mathbb{F}$ a
function such that
\begin{itemize}
\item $H(h_{1},...,h_{m})$ $\epsilon$-approximates $g$ where
$\epsilon\leq2^{-2(d+1)}$.

\item The bias of every non trivial combination of
$h_1,\ldots,h_m$ is $|\F|^{-\Omega(m+d)}$.
\end{itemize}
Then there exists a function $G:\mathbb{F}^{m}\to\mathbb{F}$ such
that $g=G(h_{1},...,h_{m})$.
\end{corollary}

We now show that $\tilde f$  can be computed by a few quadratics.

\begin{lemma}\label{lem:computinf f using few quadratics}
Let $g:\F^n \to \F$ be a quartic polynomial such that for every
$y$, $\rank_3(\Delta_y(f)) \leq \poly(1/\delta)$. Then there exist
a subspace $W$,  $c=\poly(|\F|/\delta)$ quadratics
$q'_1,\ldots,q'_c$ and a function $G$ such that, $\dim(W) = n -
\poly(|\F|/\delta)$ and $g|_W=G(q'_1,\ldots,q'_c)$.
\end{lemma}

\begin{proof}
Applying Lemma~\ref{lem:Bogdanov and Viola}, and using the fact
that every partial derivative of $g$ has a low rank, we conclude
that for $\epsilon = 2^{-20}$ there exist $c = \poly(|\F|/\delta)$
linear and quadratic functions, and a function $H$, such that
$H(\ell_1,\ldots,\ell_c,q_1,\ldots,q_c)$ $\epsilon$-approximates
$g$. Let $r = \poly(|\F|/\delta)$ and $U = \{x:
\ell_1(x)=\alpha_1,\ldots,\ell_c(x)=\alpha_c\}$ be some subspace
such that $H(\alpha_1,\ldots,\alpha_c,q_1|_U,\ldots,q_c|_U)$
$\epsilon$-approximates $g|_U$. Applying
Lemma~\ref{lem:regularizing a set of quadratics} on
$q_1|_U,\ldots,q_c|_U$ and $r$ we get that there exists a
(possible affine) subspace $W\subseteq U$ and $c' \leq c$ such
that: $\dim(W) \geq \dim(U) - (r+1)c \geq n - (r+2)c=n -
\poly(|\F|/\delta)$; w.l.o.g. for every $i=1\dots c$, $q_i|_W \in
\mathrm{span}\{q_1|_W,\ldots,q_{c'}|_W\}$; any nontrivial linear
combination of $q_1|_W,\ldots,q_{c'}|_W$ has rank larger than $r$;
${g}|_W$ is $\epsilon$-approximated by
$H(\ell_1|_W,\ldots,\ell_c|_W,q_1|_W,\ldots,q_c|_W)$ (this follows
by picking an adequate shift of the linear space in the lemma).
Hence, ${g}|_W$ is $\epsilon$-approximated by
$H(\ell_1|_W,\ldots,\ell_c|_W,q_1|_W,\ldots,q_c|_W) =
H'(q_1|_W,\ldots,q_{c'}|_W)$ for some $H'$. The reason for passing
to $W$ is that now any nontrivial linear combination of
$q_1|_W,\ldots,q_{c'}|_W$ has rank larger than $r$. We thus get by
Corollary~\ref{cor:approximating by regular} that there is some
function $G$ such that $g|_W = G(q_1|_W,\ldots,q_{c'}|_W)$.
\end{proof}

Recall that we assume w.l.o.g. that for every $y\in \F^{n-m}$,
$\rank_3(\Delta_y(\tilde f)) \leq \poly(1/\delta)$. Thus, the
lemma above implies the following corollary.

\begin{corollary}\label{cor:representing tilde f}
In the notations of the proof, there exist a subspace $Z \subset
\F^{n-m}$ of dimension $\dim(Z) \geq n - \poly(|\F|/\delta)$ such
that $\tilde{f}|_Z = F(q_1,\ldots,q_c)$, for
$c=\poly(|\F|/\delta)$ quadratic polynomials and some function
$F$.
\end{corollary}

\subsection{The structure of $f$}

We now show that we can represent $\tilde f$ as $\tilde{f} =
\sum_{i=1}^{k} \ell_i \cdot g_i + \sum_{i=1}^{k} q'_i \cdot q''_i$
where $k = \poly(|\F|/\delta)$, the $\ell_i$-s are linear, the
$q'_i$-s and $q''_i$-s are quadratic and the $g_i$-s are cubic
polynomials. For this we will transform the quadratic polynomials
to be what we denote as {\em disjoint polynomials}.

\begin{definition}\label{def:disjoint polynomials}
We say that the quadratic polynomials $\left\{ Q_{i}\right\}
_{i=1}^{m}$ are disjoint if there is a linear transformation $T$,
$2m$ variables $\left\{ x_{i}\right\} _{i=1}^{m}\cup\left\{
y_{i}\right\} _{i=1}^{m}$, where possibly for several $i$-s
$x_i=y_i$, and quadratic functions $\left\{ Q_{i}'\right\}
_{i=1}^{m}$ such that for every $k\in[m]$, $Q_{k}\circ T
=x_{k}y_{k}+Q_{k}'$ where no degree two monomial in $Q'_k$
contains a variable from $\left\{ x_{i}\right\}
_{i=1}^{m}\cup\left\{ y_{i}\right\} _{i=1}^{m}$.
\end{definition}

\begin{lemma}\label{lem:making disjoint}
Let $q_1,\ldots,q_c$ be quadratic polynomials from $\F^n$ to $\F$.
Assume that the rank of every nontrivial linear combination of
them is at least $r$. Then there exist a subspace $V
\subseteq\F^n$ of dimension $\geq n-2c^2$ and $c'\leq c$ quadratic
polynomials $q'_1,\ldots,q'_{c'}:V\to \F$ satisfying: the $q'_i$-s
are disjoint; every nontrivial linear combination of the $q'_i$-s
has rank at least $r - 2c^2$; $\mathrm{span}(q'_1,\ldots,q'_{c'})
= \mathrm{span}(q_1|_V,\ldots,q_c|_V)$.
\end{lemma}

\begin{proof}
We prove the lemma by iteratively changing each $q_i$ to a
`disjoint' form. We shall give the proof over $\F_2$ but almost
the same proof holds for odd characteristics as well. We start
with $q_1$. Assume w.l.o.g. that $x_1\cdot x_2$ appears in $q_1$.
Now, from every other $q_i$ subtract an appropriate multiple of
$q_1$ such that at the end $x_1\cdot x_2$ only appears in $q_1$.
For simplicity we call the new polynomial $q_i$ as well. Now, for
$2\leq i$ and $j \in \{1,2\}$ let $x_j\cdot \ell_{i,j}$ be the
degree two monomials involving $x_j$ in $q_i$. For $q_1$ let $x_j
\cdot \ell_{1,j}$ be the degree $2$ monomials involving $x_j$ in
$q_1 - x_1\cdot x_2$. Let $V_1 = \{x \mid \ell_{1,1}(x) = \ldots =
\ell_{2,c}(x) = 0\}$. Notice that none of the $\ell_{i,j}$-s
contain $x_1$ or $x_2$. After restricting the polynomials to $V_1$
we have that $x_1\cdot x_2$ appears in $q_1$ and every other
appearance of either $x_1$ or $x_2$ is in degree one monomials. We
now move to (the 'new') $q_2$ and continue this process. At the
end we obtain a subspace $V$ and quadratics $q'_1,\ldots,q'_{c'}$
($c'$ may be smaller than $c$ if some polynomials vanished in the
process). As at each step we set at most $2c$ linear functions to
zero, for a total of at mots $2c^2$ linear functions, the claims
about the dimension of $V$ and the rank of every linear
combination of the $q_i$-s follow. It is clear that the $q_i|_v$-s
span the $q'_i$-s and so the lemma is proved.

When dealing with odd characteristics instead of looking for
$x_1\cdot x_2$ we search for $x_1^2$. By applying an invertible
linear transformation such a monomial always exists and we
continue with the same argument.
\end{proof}

The usefulness of the definition is demonstrated in the following
lemma.

\begin{lemma}\label{lem:representing by disjoint}
Let $q_1,\ldots,q_c$ be disjoint quadratic polynomials. Assume
that $\deg(f)=2d$ and $f=F(q_1,\ldots,q_c)$ for some function
$F(z_1,\ldots,z_c)$. Then as a polynomial over $\F$, $\deg(F)\leq
d$.
\end{lemma}

\begin{proof}
We shall give the proof over $\F_2$ but it is again similar over
odd characteristic fields. Let $z_{1}^{e_1} \cdots z_{c}^{e_c}$ be
a monomial of maximal degree in $F$. When composing it with
$q_1,\ldots,q_c$ we get that $q_{1}^{e_1} \cdots q_{c}^{e_c}$
contains the monomial $\prod_{i=1}^{c} (x_i \cdot y_i)^{e_i}$. As
$z_{1}^{e_1} \cdots z_{c}^{e_c}$ is of maximal degree and each
$x_i$ and $y_i$ appear only as linear terms in all the $q_j$-s
(except the monomial $x_i\cdot y_i$ in $q_i$) we see that this
monomial cannot be cancelled by any other monomial created in
$F(q_1,\ldots,q_c)$. Therefore the monomial $\prod_{i=1}^{c} (x_i
\cdot y_i)^{e_i}$ belongs to $f$ as well. Since $\deg(f)=2d$ it
must be the case that $2e_1+\ldots+2e_c \leq 2d$. Hence,
$\deg(F)=\sum_{i=1}^{c}e_i \leq d$.
\end{proof}

We are now ready to complete the proof of
Theorem~\ref{thm:intro:deg-4:bias}.

\begin{proof}[Proof of Theorem~\ref{thm:intro:deg-4:bias}]

Combining Corollary~\ref{cor:representing tilde f},
Lemma~\ref{lem:representing by disjoint} and Lemma~\ref{lem:making
disjoint} we get that for the subspace $Z$ of
Corollary~\ref{cor:representing tilde f}, there exist a subspace
$Z' \subseteq Z$, of dimension $\dim(Z') \geq \dim(Z) -
\poly(|\F|/\delta)$, $b=\poly(|\F|/\delta)$ quadratic polynomials
$Q_1,\ldots,Q_b$ and a quadratic polynomial $H$ such that
$\tilde{f}|_{Z'} = H(Q_1,\ldots,Q_b)$. In other words
$\tilde{f}|_{Z'} = \sum_{i\leq j}\alpha_{i,j}Q_i Q_j + Q_0$.

As $f|_{Z'} = \tilde{f}|_{Z'}$ it follows that $f|_{Z'}
=\sum_{i\leq j}\alpha_{i,j}Q_i Q_j+Q_0$. Assume
w.l.o.g.\footnote{This is true up to an invertible linear
transformation and an affine shift and has no real effect on the
result, but rather simplifies the notations.} that $Z'$ is defined
as $Z' = \{x \mid x_1=\beta_1, \ldots,x_k=\beta_k\}$ for some
$k=\poly(|\F|/\delta)$. Then it is clear that we can write
$f=\sum_{i=1}^{k}x_i \cdot g_i +\sum_{i\leq j}\alpha_{i,j}Q_i Q_j
+ g_0$ for cubic polynomials $g_0,\ldots,g_k$.

\end{proof}

\section{Quartic polynomials with high $U^4$ norm}\label{sec:U4 norm}

In this section we prove Theorems~\ref{thm:intro:deg-4:gowers} and
\ref{thm:intro:deg-4:gowers:high char}. Intuitively, the notion of
$d+1$ Gowers norm indicates how close a given function is to a
degree $d$ polynomial. In fact, it was conjectured that if the
$U^{d+1}$ norm is bounded away from zero then the function has a
noticeable correlation with a degree $d$ polynomial. This
conjecture turned to be false even when the function is a degree
four polynomial and $d=3$ \cite{LMS08,GreenTao07}. Here we will
show that for this special case a weaker conclusion holds. Namely,
that for any degree four polynomial $f$ there exists a subspace of
dimension $n/\exp(\poly(1/\|f\|_{U^4}))$ on which $f|_V$ is equal
to some cubic polynomial. In fact an even stronger conclusion
holds - there exists a partition of (a subspace of small co
dimension of) $\F^n$ to such subspaces on which $f$ equals a
cubic. To ease the reading we restate
Theorem~\ref{thm:intro:deg-4:gowers} here.

\begin{theoremNoNum}[Theorem~\ref{thm:intro:deg-4:gowers}]
Let $\F$ be a finite field and $f \in \F[x_1,\ldots,x_n]$ a degree
four polynomial such that $\|f\|_{U^4}=\delta$. Then there exists
a partition of a subspace $V\subseteq \mathbb{F}^{n}$, of
dimension $\dim(V)=n-\poly(|\F|/\delta)$, to subspaces $\left\{
V_{\alpha}\right\}_{\alpha\in I}$, satisfying $\dim(V_\alpha)
=\Omega(n/|\F|^{\poly(1/\delta)})$, such that for every $\alpha\in
I$, $f|_{V_{\alpha}}$ is a cubic polynomial.
\end{theoremNoNum}

In other words, the theorem says that for $r=\poly(1/\delta)$ any
such $f$ (possibly after a change of basis of $\F^n$) can be
written as $f = \sum_{i=1}^{r}x_{n-r+i}g_i(x_1,\ldots,x_n) +
f'(x_1,\ldots,x_{n-r})+g_0$, where the $g_i$-s are degree three
polynomials and $f'$ is a polynomial for which there exists a
partition of $\mathbb{F}^{n-r}$ to subspaces $\left\{
V_{\alpha}\right\} _{\alpha\in I}$, satisfying $\dim(V_\alpha)
=\Omega(n/\exp(\poly(1/\delta)))$, such that for every $\alpha\in
I$, $f'|_{V_{\alpha}}$ is a cubic polynomial.

As in the proof of Theorem~\ref{thm:intro:deg-3:gowers} we start
by passing to a subspace of a constant codimension on which every
derivative has low rank, i.e
$\Delta_y(f)=\sum_{i=1}^{r}\ell_{i}Q_{i}+Q_{0}$. Then we shall
deduce that there is some common `basis' $\left\{ \ell_{i}\right\}
_{i=1}^{t_{2}},\left\{ Q_{i}\right\} _{i=1}^{t_{1}}$ to all the
derivatives. Namely, every derivative $\Delta_{y}(F)$ can be
expressed as
$\sum_{i=1}^{t_{1}}\ell_{i}^{y}Q_{i}+\sum_{i=1}^{t_{2}}\ell_{i}Q_{i}^{y}+Q_0^{y}$
(where $y$ in the exponent means that the polynomial may depend on
$y$). This is the main technical difficulty of the proof and it is
based on an extension of Lemma~\ref{lem:low rank space have few
common functions} to the case of low rank cubic polynomials. Then,
we conclude that for every setting $\alpha$ of $\left\{
\ell_{i}\right\} _{i=1}^{t_{2}},\left\{ Q_{i}\right\}
_{i=1}^{t_{1}}$ we obtain a subspace $V_{\alpha}$ on which all the
derivative are quadratic polynomials, i.e $f|_{V_{\alpha}}$ is
cubic.

\subsection{The case of the symmetric polynomial}
\label{sec:S_4}

Let $S_k(x_1,\ldots,x_n) = \sum_{1\leq i_1<\ldots<i_k \leq
n}x_{i_1}\cdot x_{i_2}\cdots x_{i_k}$. In \cite{GreenTao07,LMS08}
it was shown that over $\F_2$, it holds that $\|S_4\|_{U^4} \geq
\delta$, for some absolute constant $0<\delta$, but for every
degree three polynomial $g$, $\Pr[S_4 = g]\leq 1/2 + \exp(-n)$. To
make the claim of Theorem~\ref{thm:intro:deg-4:gowers} clearer we
shall work out the case of $S_4$ as an example.

Consider a partial derivative $\Delta_y(S_4)$. For simplicity
assume that $n=4m$. Computing we get that
\begin{equation}\label{eq:der of S4} \Delta_y(S_4) =
S_2 \cdot \sum_{i\neq j}^{n} x_i y_j + S_1 \cdot \sum_{i\neq
j}^{n} x_i y_j + \sum_{i\neq j}^{n} x_i y_j \;.
\end{equation}
In particular, $S_2$ is a
`basis' for the set of partial derivatives of $S_4$. Continuing,
we have that
\begin{equation}\label{eq:S2} S_2(x_1,\ldots,x_n)=
\sum_{k=1}^{2m} \left(\sum_{i=1}^{2k-1} x_i \right)\cdot
\left(x_{2k}+ \sum_{i=1}^{2k-2} x_i \right) +
\sum_{i=1}^{m}(x_{4i-3} + x_{4i-2})\;.
\end{equation}
For $k=1,\ldots,2m$ let $\ell_k = \sum_{i=1}^{2k-1} x_i$. Notice
that fixing $\ell_1,\ldots,\ell_{2m}$ reduces the degree of $S_2$
to one and so every partial derivative of $S_4$ will have degree
two. For example, consider the space $V_0=\{x\mid
\ell_1(x)=\ldots=\ell_{2m}(x)=0\}$. Rewriting we get $V_0 =
\{(0,y_1,y_1,y_2,y_2,\ldots,y_{2m-1}y_{2m-1},y_{2m})\}$. Computing
we get that
$$S_4|_{V_0} = S_2(y_1,\ldots,y_{2m-1})\;.$$
A closer inspection shows that no matter how we set
$\ell_1,\ldots,\ell_{2m}$ we will get that the degree of $S_4$
becomes two.

\subsection{Finding a `basis' for a space of low rank cubic polynomials}

In this section we prove the main technical result showing that a
subspace of degree 3 polynomials with low rank has a small
`basis'.

\begin{lemma}[Main Lemma]
\label{lem:low rank few com fun deg 4}Let $M$ be a vector space of
cubic polynomials satisfying $\rank_{3}(f)\leq r$ for all $f\in
M$. Then there exists a set of linear and quadratic functions
$\left\{ Q_{i}\right\} _{i=1}^{t_{1}}\cup\left\{ \ell_{i}\right\}
_{i=1}^{t_{2}}$, for $t_{1}\leq r$ and $t_{2}=2^{O(r)}$, such that
every $f\in M$ can be represented as
$f=\sum_{i=1}^{t_{1}}\ell_{i}^{f}Q_{i}+\sum_{i=1}^{t_{2}}\ell_{i}Q_{i}^{f}+Q_{0}^{f}$
for some linear and quadratic functions $\{ \ell_{i}^{f}\}
_{i=1}^{t_{1}}\cup\{ Q_{i}^{f}\} _{i=0}^{t_{2}}$.
\end{lemma}

The rest of this section is devoted to proving this lemma.
Similarly to the proof of Lemma~\ref{lem:low rank space have few
common functions} we will work modulo a collection of linear and
quadratic polynomials. For this we shall need the following
definition.

\begin{definition}
For a cubic polynomial $f$ we say that $\rank_3^c(f)=r$ if $r$ is
the minimal integer such that $f$ can be written as
\begin{equation}\label{eq:def of rank 3 c}
f=\sum_{i=1}^{r}\ell_{i}Q_{i}+\sum_{i=1}^{c}\ell_{i}^{(1)}\ell_{i}^{(2)}\ell_{i}^{(3)}+Q_{0}\;,
\end{equation}
where the $\ell$-s are linear functions and the $Q$-s are
quadratics.
\end{definition}

To see that difference from the previous notion of $\rank_3$
(Definition~\ref{def:cubic rank}) we observe that if $f$ is a
degree three polynomial with $\rank_{3}(f)=r$ then
$f=\sum_{i=1}^{r}\ell_iQ_i + Q_0$. If we also know that some
nontrivial linear combination of $Q_1,\ldots,Q_r$ has rank (as a
quadratic polynomial) less than $c$ then $\rank_{3}^c(f)<r$. I.e.
$\rank_{3}^{c}(f)$ ignores, in some sense,  low rank quadratic
functions in the representation of $f$.

\begin{definition}
Let $A=\left\{ Q_{i}\right\} _{i=1}^{t_{1}}\cup\left\{
\ell_{i}\right\} _{i=1}^{t_{2}}$ be a set of linear and quadratic
functions and let $f:\mathbb{F}^{n}\rightarrow\mathbb{F}$ be a
degree three polynomial. Denote \[ [f]_{A}\triangleq\left\{
f+\sum_{i=1}^{t_{1}}\ell_{i}'Q_{i}+\sum_{i=1}^{t_{2}}\ell_{i}Q_{i}'+Q_{0}'\mid
\;\mathrm{for \; linear \; and \; quadratic\; functions}\; \left\{
\ell_{i}'\right\} _{i=1}^{t_{1}},\; \left\{ Q_{i}'\right\}
_{i=0}^{t_{2}}\right\}
\]  For a linear space
$M$ of degree three functions, we define the subspace $[M]_{A}$ to
be
\[
[M]_{A}\triangleq\left\{ [f]_{A}\mid f\in M\right\} \;.\] As
before we define $\rank_{3}^{c}([f]_{A})$ to be the lowest rank of
functions in $[f]_{A}$.
\[
\rank_{3}^{c}([f]_{A})\triangleq\min\left\{ \rank_{3}^{c}(g)\mid
g\in[f]_{A}\right\} \;.\]
\end{definition}

The definition of $[f]_A$ resembles, in some sense, the notion of
working modulo an ideal. However, we note that as opposed to the
usual definition, where for every $f$, $\left\{ Q_{i}'\right\}
_{i=1}^{t_{1}}\cup\left\{ \ell_{i}'\right\} _{i=1}^{t_{2}}$ can be
arbitrary functions, in our definition they are restricted to
being quadratic and linear functions, respectively.

We are now ready to prove the main lemma of this section that
shows the existence of a small `basis' for any linear space of
cubic polynomials of low rank.

\begin{lemma}
\label{lem:low rank -> low rank 3c}Let $A=\left\{ Q_{i}\right\}
_{i=1}^{t_{1}}\cup\left\{ \ell_{i}\right\} _{i=1}^{t_{2}}$ be a
set of linear and quadratic polynomials. Let M be a linear space
of cubic polynomials such that for every $[f]_{A}\in[M]_{A}$,
$\rank_{3}^{c}([f]_{A})\leq r$. Then, there are $r$ linear
functions $\left\{ \ell_{i}\right\} _{i=1}^{r}$ and a quadratic
polynomial $Q$ such that for $A'\triangleq A\cup\left\{
\ell_{i}\right\} _{i=1}^{r}\cup\left\{ Q\right\}$ it holds that
every $[f]_{A'}\in[M]_{A'}$ satisfies
$\rank^{c'}_{3}([f]_{A'})\leq r-1$, for $c'=11c+3r+t_{1}$.
\end{lemma}

In other words, the lemma says that we can find a small set of
linear functions and one quadratic polynomial such that by adding
them to $A$ and increasing $c$ by a constant factor, we can
decrease the $\rank_3^{c'}$ of every polynomial in $[M]_{A'}$.

\begin{proof}
Assume that there is some $[g]_{A}\in[M]_{A}$ such that
$\rank_{3}^{c'}(g)=\rank_{3}^{c'}([g]_{A})=r$. If no such $g$
exists then for every $[f]_{A}\in[M]_{A}$,
$\rank^{c'}_{3}([f]_{A})\leq r-1$ and there is nothing to prove.
As $c<c'$ it also holds that $\rank_{3}^{c}([g]_{A})=r$. Hence,
$g$ can be represented as $\sum_{i=1}^{r}\ell_{i}^{g}Q_{i}^{g}+
\sum_{i=1}^{c}\ell_{i}^{g,(1)}\ell_{i}^{g,(2)}\ell_{i}^{g,(3)}$.\footnote{By
definition of $[g]_A$ we can add any quadratic polynomial to $g$
so we can assume that there is no extra $Q_0^g$ term in the
representation of $g$.} Note that $\rank_2([Q_1^g]_A)
> c'-c$ as otherwise we
could replace $Q_1^g$ with a function of the form
$\sum_{i=1}^{t_1}\alpha_i Q_i + \sum_{i=1}^{t_2}\ell_i \ell'_i +
\sum_{j=1}^{c'-c}m_im'_i$, where the $m$-s are linear functions,
and get that $\rank_3^{c'}([g]_{A})\leq r-1$.

Set $A' \eqdef A\cup\left\{ \ell_{i}^{g}\right\}
_{i=1}^{r}\cup\left\{ Q_{1}^{g}\right\}$. Assume for
contradiction that there is some $h \in M$ satisfying
$\rank^{c'}_{3}([h]_{A'})=r$. This implies that
$\rank^{c}_3([h]_{A})=r$ and that $\rank^{c'-c}_{3}([h+g]_{A'})=r$
as well. Indeed, if the latter equation was not true then by
expressing $h+g$ as a low $\rank_3^{c'-c}$ polynomial and moving
$g$ to the other side we would get that $\rank^{c'}_{3}([h]_{A'})<
r$ in contradiction (recall that $\{\ell_i^g\}\subset A'$). From
this we get that $\rank^{c}_{3}([h+g]_{A})=r$ as well. Let
$f\in[h+g]_{A}$ be such that $\rank^{c}_{3}(f)=r$. Express $h$ and
$f$ as
$h=\sum_{i=1}^{r}\ell_{i}^{h}Q_{i}^{h}+\sum_{i=1}^{c}\ell_{i}^{h,(1)}\ell_{i}^{h,(2)}\ell_{i}^{h,(3)}$
and
$f=\sum_{i=1}^{r}\ell_{i}^{f}Q_{i}^{f}+\sum_{i=1}^{c}\ell_{i}^{f,(1)}\ell_{i}^{f,(2)}\ell_{i}^{f,(3)}$.
Note that we can assume that w.l.o.g $\left\{ \ell_{i}\right\}
_{i=1}^{t_{2}}$ are linearly independent as otherwise we can just
replace them with a linearly independent subset. Similarly, we can
also assume that $\left\{ \ell_{i}\right\}
_{i=1}^{t_{2}}\cup\left\{ \ell_{i}^{g}\right\} _{i=1}^{r}$ are
linearly independent as otherwise we can find a representation for
a function in $[g]_{A}$ with a  smaller rank. Using the same
argument again we conclude that $\left\{ \ell_{i}\right\}
_{i=1}^{t_{2}}\cup\left\{ \ell_{i}^{g}\right\}
_{i=1}^{r}\cup\left\{ \ell_{i}^{h}\right\} _{i=1}^{r}$ are
linearly independent as well (by considering $[h]_{A'}$).

Since $g+h-f\in[0]_{A}$, we can express this polynomial as
$g+h-f=\sum_{i=1}^{t_{1}}\ell_{i}'Q_{i}+\sum_{i=1}^{t_{2}}\ell_{i}Q_{i}'+Q_{0}'$.
In other words:
\begin{eqnarray}\label{eq:g+h-f}
\nonumber \sum_{i=1}^{r}\ell_{i}^{g}Q_{i}^{g}
+\sum_{i=1}^{c}\ell_{i}^{g,(1)}\ell_{i}^{g,(2)}\ell_{i}^{g,(3)} +
\sum_{i=1}^{r}\ell_{i}^{h}Q_{i}^{h}
+ \sum_{i=1}^{c}\ell_{i}^{h,(1)}\ell_{i}^{h,(2)}\ell_{i}^{h,(3)} &-&\\
\left(\sum_{i=1}^{r}\ell_{i}^{f}Q_{i}^{f}
+\sum_{i=1}^{c}\ell_{i}^{f,(1)}\ell_{i}^{f,(2)}\ell_{i}^{f,(3)}
+\sum_{i=1}^{t_{1}}\ell_{i}'Q_{i}
+\sum_{i=1}^{t_{2}}\ell_{i}Q_{i}'+Q_{0}' \right)&=&0\;.
\end{eqnarray}
To ease notations, using the fact that $\left\{ \ell_{i}\right\}
_{i=1}^{t_{2}}\cup\left\{
\ell_{i}^{g}\right\}_{i=1}^{r}\cup\left\{
\ell_{i}^{h}\right\}_{i=1}^{r}$ are linearly independent, let us
assume w.l.o.g. that $\forall i$, $\ell_i^{g} = x_i$,
$\ell_i^{h}=x_{r+i}$ and $\ell_i = x_{2r+i}$. Thus,
Equation~\eqref{eq:g+h-f} becomes
\begin{eqnarray}\label{eq:g+h-f new} \nonumber
\sum_{i=1}^{r}x_{i}Q_{i}^{g}
+\sum_{i=1}^{c}\ell_{i}^{g,(1)}\ell_{i}^{g,(2)}\ell_{i}^{g,(3)} +
\sum_{i=1}^{r}x_{r+i}Q_{i}^{h}
+ \sum_{i=1}^{c}\ell_{i}^{h,(1)}\ell_{i}^{h,(2)}\ell_{i}^{h,(3)} &-&\\
\left(\sum_{i=1}^{r}\ell_{i}^{f}Q_{i}^{f}
+\sum_{i=1}^{c}\ell_{i}^{f,(1)}\ell_{i}^{f,(2)}\ell_{i}^{f,(3)}
+\sum_{i=1}^{t_{1}}\ell_{i}'Q_{i}
+\sum_{i=1}^{t_{2}}x_{2r+i}Q_{i}'+\tilde{Q}_{0} \right)&=&0\;,
\end{eqnarray}
where we remember that variables from $\{x_i\}_{i=1}^{2r+t_2}$ may
appear in the linear and quadratic functions in the expression.
Consider all terms involving $x_1$ (recall that $\ell_1^{g}=x_1$)
in Equation~\eqref{eq:g+h-f new}. Clearly they sum to zero, but
they can also be written as
\begin{eqnarray}\label{eq:g+h-f x1}\nonumber
0= Q_1^{g} + \sum_{i=1}^{r}x_{i}m_{i}^{g}
+\sum_{i=1}^{3c}\alpha_i^{g} m_{i}^{g,(1)}m_{i}^{g,(2)} +
\sum_{i=1}^{r}x_{r+i}m_{i}^{h} +
\sum_{i=1}^{3c}\alpha_i^{h}m_{i}^{h,(1)}m_{i}^{h,(2)} -\\
\left(\sum_{i=1}^{r}\beta_{i}^{f}Q_{i}^{f}+\sum_{i=1}^{r}\ell_i^{f}m_i^{f}
+\sum_{i=1}^{3c}{\beta'}_{i}^{f}m_{i}^{f,(1)}m_{i}^{f,(2)}
+\sum_{i=1}^{t_{1}}\beta_{i}'Q_{i} + \sum_{i=1}^{t_1}\ell_{i}'m_i
+\sum_{i=1}^{t_{2}}x_{2r+i}m_{i}'+m_{0} \right),
\end{eqnarray}
where the $m$-s are linear functions and the $\alpha$-s and
$\beta$-s are field elements. Rearranging terms we conclude that
\begin{equation}\label{eq:Q1g gives a contradiction}
\rank_2\left(Q_1^{g} -\sum_{i=1}^{r}\beta_{i}^{f}Q_{i}^{f} -
\sum_{i=1}^{t_{1}}\beta_{i}'Q_{i}
-\sum_{i=1}^{t_{2}}x_{2r+i}m_{i}'\right) \leq 3r + 9c +
t_1=c'-2c\;.
\end{equation}
This implies that
\begin{equation*}
\rank_2\left(\left[\sum_{i=1}^{r}\beta_{i}^{f}Q_{i}^{f}\right]_{A'}\right)\leq
c'-2c\;.
\end{equation*}
We now have two cases to consider. If
$(\beta_1^{f},\ldots,\beta_r^{f})$ are not all zero then, by
arguments described above, this implies that
$\rank^{c'-c}_{3}([f]_{A'})\leq r-1$. Recalling that
$[h+g]_{A'}=[f]_{A'}$ we get a contradiction. If, on the other
hand, $(\beta_1^{f},\ldots,\beta_r^{f})=0$ then
Equation~\eqref{eq:Q1g gives a contradiction} implies that
$\rank_2([Q_1^g]_A) \leq c'-2c$ and so $\rank_3^{c'-c}([g]_A)\leq
r-1$ in contradiction to the choice of $g$. Concluding, we have
that for every $f \in M$, $\rank_3^{c'}([f]_{A'})\leq r-1$ as
required.
\end{proof}

By applying Lemma~\ref{lem:low rank -> low rank 3c} $r$ times we
obtain the following corollary.

\begin{corollary}
\label{cor:rank c is zero}Let $M$ be a vector space of cubic
polynomials satisfying $\rank_{3}(f)\leq r$ for every $f\in M$.
Then there exists a set of quadratic and linear functions
$A=\left\{ Q_{i}\right\} _{i=1}^{r}\cup\left\{ \ell_{i}\right\}
_{i=1}^{r(r-1)/2}$, such that for $c=\exp(r)$,
$\rank^{c}_{3}([f]_{A})=0$ for every $f\in M$.
\end{corollary}

We now have that every function in $M$, modulo some set $A$ of
linear and quadratic functions, can be expressed as
$\sum_{i=1}^{c}\ell_{i}^{(1)},\ell_{i}^{(2)},\ell_{i}^{(3)}$, for
some $c$. Next we show that we can add $3c$ additional linear
functions to $A$ such that modulo the new set every function
becomes zero. We again give an iterative procedure for finding
those linear functions.

Before proving this result we define the notion of
$\dim_3^c([f]_{A})$ that will serve as a potential function in our
argument (in a similar way to the role played by $\rank_3^c$).

\begin{definition}
Let $A$ be a set of quadratic and linear functions and $[f]_{A}$ a
class of cubic functions such that $\rank_3^c([f]_A)=0$. We define
the dimension of the class as follows:
\[
\dim_3^c([f]_A)=\min\left\{ \dim\left( \mathrm{span}\left\{
\ell_{i}^{(1)},\ell_{i}^{(2)},\ell_{i}^{(3)}\right\}
_{i=1}^{c}\right)\mid
\sum_{i=1}^{c}\ell_{i}^{(1)}\ell_{i}^{(2)}\ell_{i}^{(3)} \in
[f]_A\right\}\,.
\]
\end{definition}

To better understand the reason for the definition we note that if
$\rank_3^c([f]_A)=0$ then
$\sum_{i=1}^{c}\ell_{i}^{(1)}\ell_{i}^{(2)}\ell_{i}^{(3)} + Q \in
[f]_A$ for some linear functions and quadratic $Q$. Thus, our goal
will be to find a small set of linear functions that,
simultaneously, form a basis to all those linear functions for all
$f \in M$. The next lemma shows that by joining $\left\{
\ell_{i}^{(1)},\ell_{i}^{(2)},\ell_{i}^{(3)}\right\} _{i=1}^{c}$
from some polynomial $f$, of maximal dimension in $[M]_{A}$, to
$A$, the dimension of every other element in $[M]_{A}$ decreases.

\begin{lemma}
\label{lem:low dim c implie lower dim}Let $A=\left\{ Q_{i}\right\}
_{i=1}^{t_{1}}\cup\left\{ \ell_{i}\right\} _{i=1}^{t_{2}}$ be a
set of linear and quadratic functions. Assume that the rank of any
nontrivial linear combination of $\left\{ Q_{i}\right\}
_{i=1}^{t_{1}}$ is greater than $9c+t_{1}+t_2$. Let $M$ be a
linear space of cubic polynomials such that for every
$[f]_{A}\in[M]_{A}$, $\rank^{c}_{3}([f]_{A})=0$ and
$\dim_3^c([f]_{A})\leq d$. Then, there are $d$ linear functions
$\left\{ \ell'_{i}\right\} _{i=1}^{d}$ such that for $A'\triangleq
A\cup\left\{ \ell'_{i}\right\} _{i=1}^{d}$,
$\dim_3^c([f]_{A'})\leq d-1$ for all $[f]_{A'}\in[M]_{A'}$.
\end{lemma}

The proof is very similar in nature to the proof of
Lemma~\ref{lem:low rank -> low rank 3c}.

\begin{proof}

We start by passing to the subspace $V= \{x \mid
\ell_1(x)=\ldots=\ell_{t_2}(x)=0\}$. When restricting the $Q_i$-s
to $V$ the rank of every linear combination can drop by at most
$t_2$ so it is still at least $9c+t_1$. From now on we shall work
over $V$. Note that if we prove the theorem over $V$ then it
clearly holds over $\F^n$ as well.

Let $[g]_{A}\in[M]_{A}$ be a class satisfying
$\dim_3^c([g]_{A})=d$. By definition we can assume that $g$ is
such that
$g=\sum_{i=1}^{c}\ell_{i}^{g,(1)}\ell_{i}^{g,(2)}\ell_{i}^{g,(3)}$,
and that for some $d$ linearly independent linear functions
$\left\{ \ell_{i}^{g}\right\} _{i=1}^{d}$ it holds that
$\left\{\ell_{i}^{g,(1)},\ell_{i}^{g,(2)},\ell_{i}^{g,(3)}\right\}
_{i=1}^{c}\subseteq\mathrm{span}\left\{ \ell_{i}^{g}\right\}
_{i=1}^{d}$. Set $A'=A\cup\left\{ \ell_{i}^{g}\right\}
_{i=1}^{d}$. We will show that for every $f \in M$ it holds that
$\dim_3^c([f]_{A'})\leq d-1$.

Assume for contradiction that there is some $[h]_{A'}\in[M]_{A'}$
such that $\dim_3^c([h]_{A'})=d$. Clearly, $\dim_3^c([h]_{A})=d$
as well. W.l.o.g. let
$h=\sum_{i=1}^{c}\ell_{i}^{h,(1)}\ell_{i}^{h,(2)}\ell_{i}^{h,(3)}$.
We also denote with $\left\{ \ell_{i}^{h}\right\} _{i=1}^{d}$ a
basis for
$\left\{\ell_{i}^{h,(1)},\ell_{i}^{h,(2)},\ell_{i}^{h,(3)}\right\}
_{i=1}^{c}$. As $\dim_3^c([h]_{A})$ does not decreases modulo
$\left\{ \ell_{i}^{g}\right\} _{i=1}^{d}$, it follows that
$\left\{ \ell_{i}^{g}\right\} _{i=1}^{d}\cup\left\{
\ell_{i}^{h}\right\} _{i=1}^{d}$ are linearly independent. By
definition of $A'$ we have that
$\dim_3^c([g+h]_{A'})=\dim_3^c([h]_{A'})=d$. Let $f\in[g+h]_{A}$
be such that
$f=\sum_{i=1}^{c}\ell_{i}^{f,(1)}\ell_{i}^{f,(2)}\ell_{i}^{f,(3)}$
and $\dim(\mathrm{span}\{\ell_i^{f,(j)}\})=d$. Since
$g+h-f\in[0]_{A}$ we have that $g+h-f =
\sum_{i=1}^{t_{1}}Q_{i}\ell_{i}' + Q'$. We now show that all the
$\ell'_i$-s are zero. Assume for contradiction that this is not
the case. Namely, $\left\{ \ell_{i}'\right\} _{i=1}^{t_{1}}$ are
not all zero. In particular, some $\ell'_{i}$ depends on some
variable $x$. Write $g+h-f = xF+H$ where $H$ does not depend on
$x$. We now estimate $\rank_2(F)$. On the one hand $F$ can be
expressed as
$\sum_{i=1}^{t_{1}}\alpha_{i}Q_{i}+\sum_{i=1}^{t_{1}}m_{i}\ell_{i}'+m_0$
for some coefficients $\left\{ \alpha_{i}\right\} _{i=1}^{t_{1}}$
(not all of them are zero) and some linear functions $\left\{
m_{i}\right\}_{i=0}^{t_{1}}$. Hence, $\rank_2(F)$ is larger than
$9c$ (remember that
$\rank_2(\sum_{i=1}^{t_{1}}\alpha_{i}Q_{i})>9c+t_{1}$ on $V$). On
the other hand, $g+h-f$ is equal to
$$g+h-f=\sum_{i=1}^{c}\ell_{i}^{g,(1)}\ell_{i}^{g,(2)}\ell_{i}^{g,(3)}+
\sum_{i=1}^{c}\ell_{i}^{h,(1)}\ell_{i}^{h,(2)}\ell_{i}^{h,(3)}-
\sum_{i=1}^{c}\ell_{i}^{f,(1)}\ell_{i}^{f,(2)}\ell_{i}^{f,(3)}\;,$$
so $F$ can be expressed as
$\sum_{i=1}^{9c}\hat{m}_{i}\tilde{m}_{i} + \ell$, i.e it's rank is
at most $9c$, in contradiction. It follows that $g+h-f=Q$, for
some quadratic $Q$.  Thus,
\begin{equation}\label{eq:g+h=f+Q}\sum_{i=1}^{c}\ell_{i}^{g,(1)}\ell_{i}^{g,(2)}\ell_{i}^{g,(3)}+
\sum_{i=1}^{c}\ell_{i}^{h,(1)}\ell_{i}^{h,(2)}\ell_{i}^{h,(3)}=
\sum_{i=1}^{c}\ell_{i}^{f,(1)}\ell_{i}^{f,(2)}\ell_{i}^{f,(3)}
+Q\;.
\end{equation}
For simplicity, assume w.l.o.g. that for $i=1\ldots d$, $\ell_i^g
= y_i$, $\ell_i^h = z_i$. We would like to show that if
Equation~\eqref{eq:g+h=f+Q} holds then $\deg(h)=2$ in
contradiction to the choice of $h$. To further simplify notations
we assume w.l.o.g. that the $\ell_i^f$-s are linear functions in
the variables $y_1\ldots,y_d,z_1,\ldots,z_d$ (as we can set all
other variables to zero and still obtain a similar equality). In
particular, every $\ell_{i}^{f}$ can be expressed as $\ell_i^f =
\ell_{i}^{f,g}(y)+\ell_{i}^{f,h}(z)$. Hence,
Equation~\eqref{eq:g+h=f+Q} can be rewritten as $Q(y,z)+f(y,z) =
g(y)+h(z)$. Therefore, it holds that $g(y)=f(y,0)+Q(y,0)$ and
$h(y)=f(0,z)+Q(0,z)$.\footnote{We can assume w.l.o.g. that
$\ell_i^g$ and $\ell_i^h$ do not have a constant term.} In
particular, there is some representation of $g$ and $h$ as sums of
products of linear functions such that $\left\{
\ell_{i}^{f,g}(y)\right\} $ and $\left\{ \ell_{i}^{f,h}(z)\right\}
$ are their basis, respectively. By applying an invertible linear
transformation we can further assume that $\ell_{i}^{f,g}(y)=y_i$
and $\ell_{i}^{f,h}(z)=z_i$. Thus, the basis for
$\{\ell_i^{f,(j)}\}$ is $\ell_1^f = y_1+z_1,\ldots,\ell_d^f =
y_d+z_d$. As a consequence we have that $f =
\sum_{i=1}^{c}\ell_{i}^{f,(1)}(y+z)\ell_{i}^{f,(2)}(y+z)\ell_{i}^{f,(3)}(y+z)$.

Define $F:\mathbb{F}^{d}\to\mathbb{F}$ as
${F}(u)=\sum_{i=1}^{c}\ell_{i}^{f,(1)}(u)\ell_{i}^{f,(2)}(u)\ell_{i}^{f,(3)}(u)$.
Hence, $f={F}(y+z)$, $g=f(y,0)+Q(y,0)={F}(y)+Q'(y)$ and
$h={F}(z)+Q''(z)$. Thus, for every $\alpha,\beta\in\mathbb{F}^{d}$
${F}(\alpha+\beta)={F}(\alpha)+{F}(\beta)+\tilde{Q}(\alpha,\beta)$.
It is not difficult to check that if $F$ is a polynomial such that
$\deg\left(F(\alpha+\beta)-F(\alpha)-F(\beta)\right)\leq 2$ then
$\deg(F)\leq 2$. Therefore, $[h]_{A}=[{F}(z)]_{A}=[0]_{A}$
(because ${F}$ is quadratic), in contrary to the fact that
$\dim_3^c([h]_{A})=d$. We thus deduce that for every
$[h]_{A'}\in[M]_{A'}$, $\dim_3^c([h]_{A'})< d$ as required.
\end{proof}

Combining Lemma~\ref{lem:low rank -> low rank 3c} and
Lemma~\ref{lem:low dim c implie lower dim} we are now able to
prove Lemma~\ref{lem:low rank few com fun deg 4}.

\ignore{
\begin{corollary}\label{cor:basis for U4}
Let $M$ be a vector space of cubic polynomials satisfying
$\rank_{3}(f)\leq r$ for all $f\in M$. Then there are $t_1 \leq r$
quadratic functions $Q_1,\ldots,Q_{t_1}$ and $t_2 = exp(r)$ linear
functions $\ell_1,\ldots,\ell_{t_2}$ such that every $f \in M$ can
be written as $f = \sum_{i=1}^{t_2}Q_i \cdot \ell_i^f +
\sum_{i=1}^{t_2}\ell_i \cdot Q_i^f + Q_0^f$, where superscript $f$
means a (linear or quadratic) function that depends on $f$.
\end{corollary}}

\begin{proof}[Proof of Lemma~\ref{lem:low rank few com fun deg 4}]
Corollary~\ref{cor:rank c is zero} implies that there exists a set
of quadratic and linear functions $A=\left\{ Q_{i}\right\}
_{i=1}^{r}\cup\left\{ \ell_{i}\right\} _{i=1}^{r(r-1)/2}$, such
that for $c=\exp(r)$, $\rank^{c}_{3}([f]_{A})=0$ for every $f\in
M$. By Lemma~\ref{lem:regularizing a set of quadratics} we can
assume w.l.o.g. that every nontrivial linear combination of the
$Q_i$'s have rank larger than $10c$ (possibly after passing to a
subspace $V$ of dimension at least $n-\poly(c)=n-\exp(r)$ and
throwing some of the $Q_i$-s (without changing the property of
$[M]_{A}$)). By applying Lemma~\ref{lem:low dim c implie lower
dim} $d=3c$ times we get a set $A' = \{Q'_i\}_{i=1}^{t_1} \cup
\{\ell'_i\}_{i=1}^{t_2}$, for $t_1\leq r$ and $t_2 = \exp(r)$,
such that $\dim_3^c([f]_{A'})=0$ for every $[f]_{A'} \in
[M]_{A'}$. In particular, every $f \in M$ can be represented as
$f=\sum_{i=1}^{t_{1}}\ell_{i}^{f}Q'_{i}+\sum_{i=1}^{t_{2}}\ell'_{i}Q_{i}^{f}+Q_{0}^{f}$
for some linear and quadratic functions $\left\{
\ell_{i}^{f}\right\} _{i=1}^{t_{1}}\cup\left\{ Q_{i}^{f}\right\}
_{i=0}^{t_{2}}$ depending on $f$.
\end{proof}

\subsection{Completing the proof}

We can now complete the proof of
Theorem~\ref{thm:intro:deg-4:gowers}. We first give a lemma
summarizing what we have achieved so far.

\begin{lemma}\label{lem:basis for U4}
Let $f$ be a degree four polynomial with $\|f\|_{U^4}=\delta$.
Then for $r=O(\log^2(1/\delta))$ there exist a subspace $V$,
satisfying $\dim(V) \geq n - O(\log(1/\delta))$, $r$ quadratic
polynomials $Q_1,\ldots,Q_r$ and $R=\exp(r)$ linear functions
$\ell_1,\ldots,\ell_R$ such that for every $y \in V$ we have that
$\Delta_y(f|_V)= \sum_{i=1}^{r}Q_i \cdot \ell_i^y +
\sum_{i=1}^{R}\ell_i \cdot Q_i^y + Q_0^y$.
\end{lemma}

\begin{proof}
Let $f$ be a quartic function such that $||f||_{U^{4}}>\delta$. By
Lemma~\ref{lem:gowers norm of der},
Theorem~\ref{thm:intro:deg-3:gowers} and
Lemma~\ref{lem:subadditive} there is a subspace $V$, satisfying
$\dim(V) \geq n - O(\log(1/\delta))$, such that every partial
derivative of $f|_{V}$ is a cubic polynomial of rank at most $r=
O(\log^2(1/\delta))$. Let $f'=f|_V$. Lemma~\ref{lem:low rank few
com fun deg 4} gives a set $A = \{Q_i\}_{i=1}^{r} \cup
\{\ell_i\}_{i=1}^{\exp(r)}$ such that every $\Delta_y(f')$ can be
written as $\Delta_y(f')= \sum_{i=1}^{r}Q_i \cdot \ell_i^y +
\sum_{i=1}^{\exp(r)}\ell_i \cdot Q_i^y + Q_0^y$. Notice that the
lemma concerns a linear space of cubic polynomials. In our case
the linear space will be the span of all the partial derivatives
of $f'$. As for every $y,z\in V$ it holds that
$\deg\left(\Delta_y(f')+\Delta_z(f')-\Delta_{y+z}(f')\right)=2$,
we see that in order to `close' the space we only need to add
quadratic polynomials and so the assumption about the rank of the
cubic polynomials in the space does not change.
\end{proof}

\begin{proof}[Proof of Theorem~\ref{thm:intro:deg-4:gowers}]
By Lemma~\ref{lem:basis for U4} for $r=O(\log^2(1/\delta))$ there
exist $r$ quadratics $Q_1,\ldots,Q_r$ and $R=\exp(r)$ linear
functions $\ell_1,\ldots,\ell_R$ such that for every $y \in V$ we
have that $\Delta_y(f|_V)= \sum_{i=1}^{r}Q_i \cdot \ell_i^y +
\sum_{i=1}^{R}\ell_i \cdot Q_i^y + Q_0^y$.

We now wish to express each $Q_i$ in the form of
Theorem~\ref{thm:Dickson}. We have two cases. Assume first that
$\F=\F_2$. Then for every $1\leq i \leq r$ we have that $Q_i =
\sum_{i=1}^{n/2} \ell_{i,j}\cdot \ell'_{i,j} + \ell_{i,0}$. For
$\alpha \in \F^{R}$  let $V_{\alpha} = \{x \in V \mid \forall 1
\leq i \leq R,\; \ell_i(x) = \alpha_i\}$. Clearly, $\dim(V_\alpha)
\geq \dim(V)-R$. Let $f_{\alpha} = f|_{V_\alpha}$. Then for every
$y\in V_\alpha$, $\Delta_y(f_\alpha) = \sum_{i=1}^{r} Q_i|_V \cdot
{\ell'_i}^y + {Q'}_0^y$. We now repeat the following process for
each $1\leq i \leq r$. Assume that we are working over a subspace
$V_{\alpha,\beta^1,\ldots,\beta^{i-1}}$, of dimension $d_{i-1}
=\dim\left(V_{\alpha,\beta^1,\ldots,\beta^{i-1}}\right)$. Consider
$Q_{i}|_{V_{\alpha,\beta^1,\ldots,\beta^{i-1}}}$. By
Theorem~\ref{thm:Dickson} we can write
$Q_i|_{V_{\alpha,\beta^1,\ldots,\beta^{i-1}}} =
\sum_{i=1}^{d_{i-1}/2} \ell_{i,j}\cdot \ell'_{i,j} + \ell_{i,0}$.
For $\beta^i\in \F^{d_{i-1}/2}$ define
$V_{\alpha,\beta^1,\ldots,\beta^i} = \left\{x\in
V_{\alpha,\beta^1,\ldots,\beta^{i-1}} \mid \forall 1 \leq j \leq
d_{i-1}/2,\; \ell_{i,j} = (\beta^i)_j\right\}$.  Note that
$\cup_{\beta^i \in
\F^{d_{i-1}/2}}V_{\alpha,\beta^1,\ldots,\beta^i} =
V_{\alpha,\beta^1,\ldots,\beta^{i-1}}$. Thus, the set
$\{V_{\alpha,\beta^1,\ldots,\beta^r}\}$ forms a partition of $V$.
Moreover, observe that for every $\alpha,\beta^1,\ldots,\beta^i$,
$\deg\left(Q_i|_{V_{\alpha,\beta^1,\ldots,\beta^i}}\right) \leq
1$. Thus, for every $\alpha,\beta^1,\ldots,\beta^r$, all the
partial derivatives of $f|_{V_{\alpha,\beta^1,\ldots,\beta^r}}$
are of degree two and so
$\deg\left(f|_{V_{\alpha,\beta_1,\ldots,\beta_r}}\right)\leq 3$ as
claimed. To finish the proof we note that
$\dim\left({V_{\alpha,\beta^1,\ldots,\beta^i}}\right) \geq
\dim\left({V_{\alpha,\beta^1,\ldots,\beta^{i-1}}}\right)/2$.
Therefore, $\dim\left({V_{\alpha,\beta^1,\ldots,\beta^r}}\right)
\geq (n-R))/2^r = n/\exp(\log^2(1/\delta))$.

When $\chr(\F)=p>2$ we have the representation
$Q_i|_{V_{\alpha,\beta^1,\ldots,\beta^{i-1}}} =
\sum_{i=1}^{d_{i-1}} \ell_{i,j}^2 + \ell_{i,0}$. Rewriting we
obtain  \begin{eqnarray*}
Q_i|_{V_{\alpha,\beta^1,\ldots,\beta^{i-1}}}
&=&\sum_{i=1}^{r}\ell_{i,j}^{2}+\ell_{0} \\
&=&\sum_{i=1}^{d_{i-1}/p}\sum_{j=0}^{p-1}\ell_{pi+j}^{2}+\ell_{0}\\
&=&\sum_{i=1}^{d_{i-1}/p}\left(\sum_{j=1}^{p-1}\left(\ell_{pi+j}-\ell_{pi}\right)^{2}+2\ell_{pi}
\sum_{j=1}^{p-1}\left(\ell_{pi+j}-\ell_{pi}\right)\right)+\ell_{0}\end{eqnarray*}
Observe that after fixing $\forall1\leq j\leq p-1,\;
\ell_{pi+j}-\ell_{pi} = (\beta^i)_j$,
$Q_i|_{V_{\alpha,\beta^1,\ldots,\beta^{i-1}}}$ becomes linear.
Thus, the same argument as before gives the required result here
as well.
\end{proof}

Combining the idea of the above proof with the notion of disjoint
polynomials we prove Theorem~\ref{thm:intro:deg-4:gowers:high
char}.

\begin{proof}[Proof Sketch of Theorem~\ref{thm:intro:deg-4:gowers:high char}]
As in the proof of Theorem~\ref{thm:intro:deg-4:gowers} we obtain
linear $\{\ell_i\}_{i=1\ldots R}$ and quadratic
$\{q_i\}_{i=1\ldots r}$, where $r = O(\log^2(1/\delta))$ and
$R=\exp(r)$, that form a `basis' to the set of partial
derivatives. By passing to a subspace of codimension $R$ and using
Lemma~\ref{lem:making disjoint} we can assume w.l.o.g. that the
$q_i$-s are disjoint and that every partial derivative has the
form $\Delta_y(f)=\sum_{i=1}^{r}q_i \cdot \ell_i^{(y)} +
q_0^{(y)}$. As $\mathrm{char}({\F})>4$ we can assume w.l.o.g. that
$q_i = x_i^2 + q'_i$ and that $x_j$ can appear in $q'_i$ only as a
linear term. We now subtract from $f$ terms of the form $\alpha
q_i q_j$ such that in the resulting polynomial $f'$ there will be
no monomial of the form $x_i^2 x_j^2$ for $i \leq j$. Note that
$f'$ also has the property that for every $y$,
$\Delta_y(f')=\sum_{i=1}^{r}q_i \cdot {\ell'}_i^{(y)} +
{q'_0}^{(y)}$. We now show that degree four monomials in $f'$ may
only contain $x_i$ or $x_i^3$ but not $x_i^2$, for $i\in[r]$.
Indeed, assume for a contrary that $x_i^2$ appears in a degree
four monomial. Then, $x_i$ appears in $\Delta_{x_i}(f')$ in a
degree three monomial. This monomial comes from some
$\ell^{(x_i)}_j q_j$ for $j \neq i$. Therefore, we also have the
term $x_i x_j^2$ in $\Delta_{x_i}(f')$ (it is not difficult to see
that this term cannot be cancelled by any other $\ell^{(x_i)}_k
q_k$). As $\mathrm{char}(\F)>4$, integration w.r.t. $x_i$ gives
that  the term $x_i^2 x_j^2$ appears in $f'$ in contradiction. We
can thus write $f' = \sum_{i=1}^{r}x_i^3 \tilde\ell_i + f''$,
where in $f''$ each $x_i$ has degree at most one. Consider any $y$
`orthogonal' to
$\{x_1,\ldots,x_r,\tilde\ell_1,\ldots,\tilde\ell_r\}$ (namely,
substituting $y$ in any of those linear functions gives zero).
Then for each $i$, $\Delta_y(x_i^3 \tilde\ell_i)=0$. Hence, $x_i$
is the highest power of $x_i$ appearing in $\Delta_y(f')$. As the
$q_i$-s are disjoint and $\Delta_y(f')=\sum_{i=1}^{r}q_i \cdot
{\ell'}_i^{(y)} + {q'}_0^{(y)}$ we obtain that it must be the case
that $\deg(\Delta_y(f'))\leq 2$. Thus, $f'$ can be rewritten as a
polynomial in at most $2r$ variables plus a degree three
polynomial. Therefore, possibly after a change of basis we can
write $f = \sum_{i\leq j}\alpha_{i,j}q_i \cdot q_j +
\sum_{i=1}^{2r+R} y_i \cdot g_i + g_0$ as needed.
\end{proof}

\section{Conclusions}

In this paper we gave strong structural results for degree three
and four polynomials that have a high bias. It is a very
interesting question whether such a structure exists for higher
degree biased polynomials. Green and Tao \cite{GreenTao07} proved
such a result when $\deg(f)<|\F|$ (with much worse parameters for
degrees three and four), so this question is mainly open for small
fields. Another interesting question is improving the parameters
in the results of \cite{GreenTao07,KaufmanLovett08}. There it was
shown that when $\deg(f)=d$ and $f$ is biased then
$f=F(g_1,\ldots,g_{c_d})$, where $\deg(g_i)<\deg(f)$. However, the
dependence of $c_d$ on the degree $d$ and the bias $\delta$ is
terrible. Basically, $c_3 = \exp(\poly(1/\delta)$ and $c_d$ is a
tower of height $c_{d-1}$. In contrast, our results give that $c_3
= \log^2(1/\delta)$ and $c_4 = \poly(1/\delta)$. Thus, it is an
intriguing question to find the true dependence of $c_d$ on
$\delta$. In particular, as far as we know, it may be the case
that $c_d$ is polynomial in $1/\delta$ (where the exponent may
depend on $d$), or even $\poly(\log(1/\delta))$.

For the case of degree four polynomials with high $U^4$ norm we
proved an inverse theorem showing that on many subspaces, of
dimension $\Omega(n)$, $f$ equals to a degree three polynomial (a
different polynomial for each subspace). Such a result seems
unlikely to be true for higher degrees. However, it may be the
case that if $\deg(f)=d$ and $f$ has a high $U^d$ norm then $f$ is
correlated with a lower degree polynomial on a high dimensional
subspace.

\section*{Acknowledgements}

The authors would like to thank Shachar Lovett, Partha
Mukhopadhyay and Alex Samorodnitsky for helpful discussions at
various stages of this work. We are especially grateful to Shachar
and Partha for many helpful comments on an earlier version of this
paper. E.H would like to thank Noga Zewi for many helpful
conversations and for her support. Finally, we thank Swastik
Kopparty, Shubhangi Saraf and Madhu Sudan for pointing out an
error in an earlier proof of Lemma~\ref{lem:low rank space have
few common functions}.


\newcommand{\etalchar}[1]{$^{#1}$}

\end{document}